\documentclass[11pt]{article}

\usepackage{amsmath, amsthm, amsfonts,color}
\usepackage[margin=1in]{geometry}

\newtheorem{thm}{Theorem}[section]

\newtheorem{prop}[thm]{Proposition}
\theoremstyle{definition}

\theoremstyle{remark}
\newtheorem{rem}[thm]{Remark}

\def\RR{\mathbb{R}}
\def\NN{\mathbb{N}}
\def\di{\displaystyle}

\title{Anisotropic singular Neumann equations with unbalanced growth}
\author{Nikolaos S. Papageorgiou\footnote{Department of Mathematics, National Technical University, Zografou  Campus, 15780 Athens, Greece.
E-mail: {\tt npapg@math.ntua.gr}}, Vicen\c tiu D. R\u adulescu\footnote{Faculty of Applied Mathematics, AGH University of Science and Technology, 30-059 Krakow, Poland
\& Department of Mathematics, University of Craiova, 200585 Craiova, Romania.  Corresponding author. ORCID: {\tt http://orcid.org/0000-0003-4615-5537}. E-mail: {\tt radulescu@inf.ucv.ro} },  Du\v{s}an D. Repov\v s\footnote{Faculty of Education and Faculty of Mathematics and Physics, University of Ljubljana, and Institute of Mathematics, Physics and Mechanics, 1000 Ljubljana, Slovenia.
E-mail: {\tt dusan.repovs@guest.arnes.si}}}

\begin{document}
\maketitle

\begin{abstract}
We consider a nonlinear parametric Neumann problem driven by the anisotropic $(p,q)$-Laplacian and a reaction which exhibits the combined effects of a singular term and of a parametric superlinear perturbation. We are looking for positive solutions. Using a combination of topological and variational tools together with suitable truncation and comparison techniques, we prove a bifurcation-type result describing the set of positive solutions as the positive parameter $\lambda$ varies. We also show the existence of minimal positive solutions $u_\lambda^*$ and determine the monotonicity and continuity properties of the map $\lambda\mapsto u_\lambda^*$.


\end{abstract}

\section{Introduction}
This paper was motivated by several recent contributions to the qualitative analysis of nonlinear problems with unbalanced growth. We mainly refer to the pioneering contributions of Marcellini \cite{marce1,marce2,marce3} who studied lower semicontinuity and regularity properties of minimizers of certain quasiconvex integrals. Problems of this type arise in nonlinear elasticity and are connected with the deformation of an elastic body, cf. Ball \cite{ball1,ball2}.

We are concerned with the qualitative analysis of a class of anisotropic singular problems with Neumann boundary condition and driven by a differential operator with unbalanced growth. The features of this paper are the following:

(i) the problem studied in the present work is associated to a {\it double phase energy} with {\it variable exponents} (variational integral with anisotropic unbalanced growth);

(ii) the reaction is both {\it singular} and {\it anisotropic};

(iii) we assume a {\it Neumann boundary condition}.

To the best of our knowledge, this is the first paper dealing with the combined effects generated by the above features.

\subsection{Unbalanced problems and their historical traces}
Let $\Omega$ be a bounded domain in $\RR^N$ ($N\geq 2$) with a smooth boundary. If $u:\Omega\to\RR^N$ is the displacement and $Du$ is the $N\times N$  matrix of the deformation gradient, then the total energy can be represented by an integral of the type
\begin{equation}\label{paolo}I(u)=\int_\Omega f(z,Du(z))dz,\end{equation}
where the energy function $f=f(z,\xi):\Omega\times\RR^{N\times N}\to\RR$ is quasiconvex with respect to $\xi$, see Morrey \cite{morrey}. One of the simplest examples considered by Ball is given by functions $f$ of the type
$$f(\xi)=g(\xi)+h({\rm det}\,\xi),$$
where ${\rm det}\,\xi$ is the determinant of the $N\times N$ matrix $\xi$, and $g$, $h$ are nonnegative convex functions, which satisfy the growth conditions
$$g(\xi)\geq c_1\,|\xi|^p;\quad\lim_{t\to+\infty}h(t)=+\infty,$$
where $c_1$ is a positive constant and $1<p<N$. The condition $p\leq N$ is necessary to study the existence of equilibrium solutions with cavities, that is, minima of the integral \eqref{paolo} that are discontinuous at one point where a cavity forms; in fact, every $u$ with finite energy belongs to the Sobolev space $W^{1,p}(\Omega,\RR^N)$, and thus it is a continuous function if $p>N$. In accordance with these problems arising in nonlinear elasticity, Marcellini \cite{marce1,marce2} considered continuous functions $f=f(x,u)$ with {\it unbalanced growth} that satisfy
$$c_1\,|u|^p\leq |f(z,u)|\leq c_2\,(1+|u|^q)\quad\mbox{for all}\ (x,u)\in\Omega\times\RR,$$
where $c_1$, $c_2$ are positive constants and $1\leq p\leq q$. Regularity and existence of solutions of elliptic equations with $p,q$--growth conditions were studied in \cite{marce2}.

The study of non-autonomous functionals characterized by the fact that the energy density changes its ellipticity and growth properties according to the point has been continued in a series of remarkable papers by Mingione {\it et al.} \cite{1Bar-Col-Min, 2Bar-Col-Min, 8Col-Min}. These contributions are in relationship with the work of Zhikov \cite{23Zhikov, 24Zhikov}, which describe the
behavior of phenomena arising in nonlinear
elasticity.
In fact, Zhikov intended to provide models for strongly anisotropic materials in the context of homogenisation.
In particular, he considered the following model
functional
\begin{equation}\label{mingfunc}
{\mathcal P}_{p,q}(u) :=\int_\Omega (|Du|^p+a(z)|Du|^q)dz,\quad 0\leq a(x)\leq L,\ 1<p<q,
\end{equation}
where the modulating coefficient $a(x)$ dictates the geometry of the composite made of
two differential materials, with hardening exponents $p$ and $q$, respectively.

In the present paper we are concerned with a problem whose energy is of the type defined in \eqref{mingfunc} but such that the exponents $p$ and $q$ are variable (they depend on the point).

\subsection{Statement of the problem}
Let $\Omega\subseteq\RR^N$ be a bounded domain with a $C^2$-boundary $\partial\Omega$. In this paper we study the following parametric singular anisotropic $(p,q)$-equation:
\begin{equation}\tag{\mbox{$P_\lambda$}}
\left\{
\begin{array}{lll}
-\Delta_{p(z)}u(z)-\Delta_{q(z)}u(z)+\xi(z)u(z)^{p(z)-1}=u(z)^{-\eta(z)}+\lambda f(z,u(z)) \text{ in } \Omega,\\
\di \frac{\partial u}{\partial n}=0 \mbox{ on }\partial\Omega,\  {u(z)>0\ \mbox{for all $z\in\Omega$}},\ \lambda>0.
\end{array}
\right.
\end{equation}

 {In this problem, we make the following hypotheses for the exponents $p(\cdot),\, q(\cdot),\, \eta(\cdot)$:
$$p,\, q,\, \eta\in C^1(\overline\Omega),\ q_-\leq q_+<p_-\leq p_+,\ 0<\eta(z)<1\ \mbox{for all}\ z\in\overline\Omega,$$
where for every $r\in C(\overline\Omega)$ we define
$$r_-:=\min_{\overline\Omega}r,\ r_+:=\max_{\overline\Omega}r.$$}

Also for $r\in C(\overline{\Omega})$ with $1<r(z)<\infty$ for all $z\in \overline{\Omega}$, we denote by $\Delta_{r(z)}$  the $r(z)$-Laplace differential operator defined by
$$
\Delta_{r(z)} u={\rm div} \,(|Du|^{r(z)-2}Du) \mbox{ for all }u\in W^{1,r(z)}(\Omega).
$$

The potential function $\xi\in L^{\infty}(\Omega)$ satisfies $\xi(z)>0$ for a.a. $z\in\Omega$. In the reaction we have two terms. One is the singular term $x\mapsto x^{-\eta(z)}$ with $0<\eta(z)<1$ for all $z\in\overline{\Omega}$ and the other is a parametric perturbation $\lambda f(z,x)$ with $\lambda>0$ being the parameter.  {The function $f(z,x)$ is a Carath\'eodory function, that is, for all $x\in\RR$ the mapping $z\mapsto f(z,x)$ is measurable and for a.a. $z\in\Omega$ the function $x\mapsto f(z,x)$ is continuous.} We assume that for a.a. $z\in\Omega$, the function $f(z,\cdot)$ exhibits a ($p_+-1$)-superlinear growth near $+\infty$ with $p_+=\displaystyle{\max_{\overline{\Omega}}p}$, but without satisfying the so-called Ambrosetti-Rabinowitz condition (the $AR$-condition for short), which is common in the literature when dealing with superlinear problems. Instead, we use a less restrictive condition which incorporates in our framework superlinear nonlinearities with slower growth near $+\infty$.  {The precise hypotheses on $f(z,x)$ can be found in Section 2 (see hypotheses $H_1$).}

We are looking for positive solutions and our aim is to determine how the set of positive solutions changes as the parameter $\lambda>0$ varies. In this direction we prove a bifurcation-type result describing the changes in the set of positive solutions of $(P_\lambda)$ as the positive parameter $\lambda$ increases. We also show that if $\lambda>0$ is admissible (that is, problem $(P_\lambda)$ admits positive solutions), then there is a minimal positive solution $u_\lambda^*$ (that is, a smallest solution) and we examine the monotonicity and continuity of the map $\lambda\mapsto u_\lambda^*$.

Analogous studies for $p$-Laplacian equations with constant exponent, were conducted by Giacomoni, Schindler \& Taka\v c \cite{11Gia-Sch-Tak} and Papageorgiou \& Winkert \cite{19Pap-Win}. More general equations driven by nonhomogeneous differential operators, were considered recently by Papageorgiou, R\u adulescu \& Repov\v s \cite{prrpams, prrbull, prrzamp,16Pap-Rad-Rep}, Papageorgiou \& Scapellato \cite{paps}, Papageorgiou, Vetro \& Vetro \cite{18Pap-Vet-Vet}, Papageorgiou \& Zhang \cite{papz}, and Ragusa \& Tachikawa \cite{ragusa}. 
 {We should also mention the very recent works of De Filippis \& Mingione \cite{refe1} and Marcellini \cite{refe2}, on the regularity of solutions of double phase problems. This is a very interesting area with several issues remaining open and requiring further investigation. Finally, we mention the work of Bahrouni, R\u adulescu \& Winkert \cite{refe3} on a class of double phase problems with convection.}
Singular anisotropic equations driven by the $p(z)$-Laplacian, were studied by Byun \& Ko \cite{4Byu-Ko}, Zhang \& R\u{a}dulescu \cite{zhang}, and Saudi \& Ghanmi \cite{22Sau-Gha}. To the best of our knowledge, there are no works on singular anisotropic $(p,q)$-equations.

Boundary value problems driven by a combination of differential operators (such as $(p,q)$-equations) arise in many mathematical models of physical processes. We mention the historically first such work of  Cahn \& Hilliard \cite{5Cah-Hil}, which deals with the process of separation of binary alloys and the more recent works of Benci, D'Avenia, Fortunato \& Pisani \cite{3Ben-D'Av-For-Pis} on quantum physics, of Cherfils \& Ilyasov \cite{6Che-Ily} on reaction diffusion systems and of Bahrouni, R\u adulescu \& Repov\v s \cite{1Bah-Rad-Rep,2Bah-Rad-Rep} on transonic flow problems. Boundary value problems involving differential operators with variable exponents, are studied in the book of R\u adulescu \& Repov\v s \cite{21Rad-Rep}, while a comprehensive discussion of semilinear singular problems and a rich relevant bibliography can be found in the book of Ghergu \& R\u adulescu \cite{10Ghe-Rad}.

\section{Mathematical background and hypotheses}
 {Although as we already mentioned in the previous section, we require that our exponents $p(\cdot)$, $q(\cdot)$, $\eta(\cdot)$ are smooth (in order to exploit the existing anisotropic regularity theory), the introduction of the variable exponent spaces does not require such regularity restrictions.}

We introduce the following spaces
\begin{eqnarray*}
   && M(\Omega)=\{u:\Omega\to\RR \mbox{ measurable}\} \\
   && L_1^\infty(\Omega)=\{p\in L^\infty(\Omega):\:1\leq \underset{\Omega}{{\rm essinf}}\,p\}.
\end{eqnarray*}

As usual, we identify in $M(\Omega)$ two functions which differ only on a set of measure zero. If $p\in L_1^\infty(\Omega)$, then we set
$$
p_-=\underset{\Omega}{\rm essinf}\,p\  \mbox{ and }\ p_+=\underset{\Omega}{\rm esssup}\,p.
$$

Given $p\in L_1^\infty(\Omega)$, the variable exponent Lebesgue space $L^{p(z)}(\Omega)$ is defined by
$$
L^{p(z)}(\Omega)=\left\{u\in M(\Omega):\: \int_\Omega |u|^{p(z)}dz<\infty\right\}.
$$

We equip this space with the so-called ``Luxemburg norm" defined by
$$
\|u\|_{p(z)}=\inf\left\{\lambda>0:\: \int_\Omega \left(\frac{|u|}{\lambda}\right)^{p(z)}dz\leq1\right\}.
$$

Variable exponent Lebesgue spaces are similar to the classical Lebesgue spaces. More precisely, they are separable Banach spaces, they are reflexive if and only if $1<p_-\leq p_+<\infty$ (in fact, they are uniformly convex). Moreover, simple functions and continuous functions of compact support are dense in $L^{p(z)}(\Omega)$.

Suppose that $p,q\in L_1^\infty(\Omega)$. Then we have the following property:
$$
``L^{p(z)}(\Omega)\hookrightarrow L^{q(z)}(\Omega) \mbox{ continuously}
$$
$$
\mbox{ if and only if }
$$
$$
q(z)\leq p(z) \mbox{ for a.a. }z\in \Omega".
$$

Let $p,p'\in L_1^\infty(\Omega)$ such that $\frac{1}{p(z)}+\frac{1}{p'(z)}=1$ for a.a. $z\in\Omega$. Then $L^{p(z)}(\Omega)^*=L^{p'(z)}(\Omega)$ and the following H\"older-type inequality is true
$$
\int_\Omega |uv|dz\leq\left(\frac{1}{p_-}+\frac{1}{p'_-}\right)\|u\|_{p(z)}\|v\|_{p'(z)}
$$
for all $u\in L^{p(z)}(\Omega)$ and all $v\in L^{p'(z)}(\Omega)$.

Using the variable exponent Lebesgue spaces, we can define in the usual way variable exponent Sobolev spaces. So, if $p\in L_1^\infty(\Omega)$, then we set
$$
W^{1,p(z)}(\Omega)=\{u\in L^{p(z)}(\Omega):\:|Du|\in L^{p(z)}(\Omega)\}.
$$

This space is furnished with the following norm
$$
\|u\|=\|u\|_{p(z)}+\|\, |Du|\,\|_{p(z)}.
$$

Evidently, an equivalent norm is given by
$$
|u|=\inf\left\{\lambda>0:\:\int_\Omega \left[\left(\frac{|Du|}{\lambda}\right)^{p(z)}+\left(\frac{|u|}{\lambda}\right)^{p(z)}\right]dz\leq1\right\}.
$$

The anisotropic Sobolev space $W^{1,p(z)}(\Omega)$ is a separable Banach space and if $1<p_-\leq p_+<\infty$, then $W^{1,p(z)}(\Omega)$ is reflexive (in fact, uniformly convex). Note that $W^{1,p(z)}(\Omega)\hookrightarrow W^{1,p_-}(\Omega)$ continuously. Also, $W^{1,p(z)}_0(\Omega)$ is the closure of the set of $W^{1,p(z)}(\Omega)$-functions with compact support, that is, of the set
$$
\left\{u\in W^{1,p(z)}(\Omega):\:u=u\chi_K \mbox{ with } K\subseteq\Omega \mbox{ compact }\right\}.
$$

If $p\in C^1(\overline{\Omega})$, then $W^{1,p(z)}_0(\Omega)=\overline{C_c^\infty(\Omega)}^{\|\cdot\|}$.

We set
$$
p^*(z)=\left\{
         \begin{array}{ll}
           \frac{Np(z)}{N-p(z)}, & \hbox{ if }p(z)<N \\
            +\infty, & \hbox{ if } N\leq p(z).
         \end{array}
       \right.
$$

Suppose that $p,q\in L_1^\infty(\Omega)\cap C(\overline{\Omega})$ with $p_+<N$ and $q(z)\leq p^*(z)$ (resp., $q(z)<p^*(z)$) for all $z\in\overline{\Omega}$, then we have that $W^{1,p(z)}(\Omega)\hookrightarrow L^{q(z)}(\Omega)$ continuously (resp., compactly). A comprehensive presentation of variable exponent Lebesgue and Sobolov spaces can be found in the book of Diening, Harjulehto, H\"asto \& Ruzicka \cite{7Die-Har-Has-Ruz}.

Let $r\in L_1^\infty(\Omega)$ and consider the Lebesgue space $L^{r(z)}(\Omega)$. The modular function for this space is given by
$$
\rho_{r(z)}(u)=\int_\Omega |u|^{r(z)}dz.
$$

This function is basic in the study of $L^{r(z)}(\Omega)$ and is closely related to the norm $\|\cdot\|_{p(z)}$ introduced above. More specifically, we have the following result.

\begin{prop}\label{prop1}
  \begin{itemize}
    \item[(a)] For $u\in L^{r(z)}(\Omega)$, $u\not=0$, we have
$$
\|u\|_{r(z)}\leq \lambda\;\Leftrightarrow\;\rho_{r(z)}\left(\frac{u}{\lambda}\right)\leq1;
$$
    \item[(b)] $\|u\|_{r(z)}<1$ (resp. $=1$, $>1$) $\Leftrightarrow$ $\rho_{r(z)}(u)<1$ (resp. $=1$, $>1$);
    \item[(c)] $\|u\|_{r(z)}<1$ $\Rightarrow$ $\|u\|_{r(z)}^{r_+}\leq\rho_{r(z)}(u)\leq\|u\|_{r(z)}^{r_-}$,\\
               $\|u\|_{r(z)}>1$ $\Rightarrow$ $\|u\|_{r(z)}^{r_-}\leq\rho_{r(z)}(u)\leq\|u\|_{r(z)}^{r_+}$;
    \item[(d)] $\|u_n\|_{r(z)}\to0$ $\Leftrightarrow$ $\rho_{r(z)}(u_n)\to 0$;
    \item[(e)] $\|u_n\|_{r(z)}\to+\infty$ $\Leftrightarrow$ $\rho_{r(z)}(u_n)\to+\infty$.
  \end{itemize}
\end{prop}

We consider the map $A_{r(z)}:W^{1,r(z)}(\Omega)\to W^{1,r(z)}(\Omega)^*$ defined by
$$
\langle A_{r(z)}(u),h\rangle=\int_\Omega |Du|^{r(z)-2}(Du,Dh)_{\RR^N}dz \mbox{ for all }u,h \in W^{1,r(z)}(\Omega).
$$

This map has the following properties (see Gasinski \& Papageorgiou \cite[Proposition 2.5]{9Gas-Pap}).

\begin{prop}\label{prop2}
  The map $A_{r(z)}:W^{1,p(z)}(\Omega)\to W^{1,p(z)}(\Omega)^*$ is bounded (that is, maps bounded sets to bounded sets), continuous, monotone, hence also maximal monotone  and of type $(S)_+$, that is,
$$
``u_n\overset{w}{\to}u \mbox{ in }W^{1,r(z)}(\Omega),\;\limsup_{n\to\infty}\langle A_{r(z)}(u_n),u_n-u\rangle\leq0\Rightarrow u_n\to u \mbox{ in } W^{1,r(z)}(\Omega)."
$$
\end{prop}

In addition to the variable exponent spaces, we will use the Banach space $C^1(\overline{\Omega})$. This is an ordered Banach space with positive cone $C_+=\{u\in C^1(\overline{\Omega}):\:u(z)\geq0 \mbox{ for all }z\in\overline{\Omega}\}$. This cone has nonempty interior given by
$$
{\rm int}\,C_+=\{u\in C_+:\:u(z)>0 \mbox{ for all }z\in\overline{\Omega}\}.
$$

We will also use another open cone in $C^1(\overline{\Omega})$ given by
$$
D_+=\left\{u\in C^1(\overline{\Omega}):\:u(z)>0 \mbox{ for all }z\in\Omega,\ \frac{\partial u}{\partial n}\Big|_{\partial\Omega \cap u^{-1}(0)}<0\right\},
$$
with $n(\cdot)$ being the outward unit normal on $\partial\Omega$.

Combining the proofs of Proposition 2.5 of \cite{17Pap-Rad-Rep} and of Proposition 6 in \cite{16Pap-Rad-Rep} we have the following strong comparison principle.

\begin{prop}\label{prop3}
  If $p,q,\eta\in C^1(\overline{\Omega})$, $1<q_-\leq q_+<p_-\leq p_+$, $0<\eta(z)<1$ for all $z\in\overline{\Omega}$, $\hat{\xi},h,g\in L^\infty(\Omega)$, $\hat{\xi}(z)\geq0$ for a.a. $z\in\Omega$, $0<\mu\leq g(z)-h(z)$ for a.a. $z\in\Omega$ and $u,v\in C^1(\overline{\Omega})$ satisfy $0\leq u\leq v$ and
\begin{eqnarray*}
  && -\Delta_{p(z)}u-\Delta_{q(z)}u +\hat{\xi}(z) u^{p(z)-1}-u^{-\eta(z)}=h(z) \mbox{ in }\Omega,\\
  && -\Delta_{p(z)}v-\Delta_{q(z)}v +\hat{\xi}(z) v^{p(z)-1}-v^{-\eta(z)}=g(z) \mbox{ in }\Omega,
\end{eqnarray*}
then $v-u\in D_+$.
\end{prop}

If $u,v\in W^{1,p(z)}(\Omega)$ $(p\in L_1^\infty(\Omega))$ with $u\leq v$, then we define
$$
[u,v]=\left\{h\in W^{1,p(z)}(\Omega):\:u(z)\leq h(z)\leq v(z) \mbox{ for a.a. }z\in\Omega\right\}
$$
and
$$
[u)=\left\{h\in W^{1,p(z)}(\Omega):\:u(z)\leq h(z) \mbox{ for a.a. }z\in \Omega\right\}.
$$

If $X$ is a Banach space and $\varphi\in C^1(X,\RR)$, then we denote by $K_\varphi$ the critical set of $\varphi$, that is, the set
$$
K_\varphi=\{u\in X:\:\varphi'(u)=0\}.
$$

Also, we say that $\varphi\in C^1(X,\RR)$ satisfies the ``$C$-condition", if the following property holds:
$$
``\mbox{ Every sequence $\{u_n\}_{n\geq1}\subseteq X$ such that }
\left\{\varphi(u_n)\right\}_{n\geq1}\subseteq\RR \mbox{ is bounded and}$$
$$\mbox{ and } (1+\|u_n\|_X)\varphi'(u_n)\to0 \mbox{ in }X^* \mbox{ as } n\to \infty,
\mbox{ admits a strongly convergent subsequence".}
$$

This is a compactness-type condition on the functional $\varphi(\cdot)$. It compensates for the fact that $X$ is not locally compact, being in general infinite dimensional. The $C$-condition plays a crucial role in the minimax theory of the critical values of the functional $\varphi(\cdot)$.

Now we are ready to introduce the hypotheses on the data of $(P_\lambda)$.

\smallskip
$H_0$: $p,q,\eta\in C^1(\overline{\Omega})$, $1<q_-\leq q_+<p_-\leq p_+$, $0<\eta(z)<1$ for all $z\in\overline{\Omega}$, $\xi\in L^\infty(\Omega)$, $\xi(z)>0$ for a.a. $z\in\Omega$.

\smallskip
$H_1$: $f:\Omega\times\RR\to \RR$ is a Carath\' eodory function (that is, $f(z,x)$ is measurable in $z\in\Omega$ and continuous in $x\in\RR$) such that $f(z,0)=0$ for a.a. $z\in \Omega$ and
\begin{itemize}
  \item[$(i)$] $0\leq f(z,x)\leq a(z)[1+x^{r(z)-1}]$ for a.a. $z\in\Omega$, all $x\geq0$ with $a\in L^\infty(\Omega)$ and $r\in C(\overline{\Omega})$ with $p_+<r_-\leq r_+<p^*(z)$ for all $z\in\overline{\Omega}$;
  \item[$(ii)$] if $F(z,x)=\displaystyle{\int_0^x f(z,s)ds}$,
then $\displaystyle{\lim_{x\to+\infty}\frac{F(z,x)}{x^{p_+}}=+\infty}$ uniformly for a.a. $z\in \Omega$;
  \item[$(iii)$] if $\xi_\lambda(z,x)=\left(1-\frac{p_+}{1-\eta(z)}\right)x^{1-\eta(z)}+\lambda\left[f(z,x)x-p_+F(z,x)\right]$,
then there exists $\hat{\vartheta}_\lambda\in L^1(\Omega)$ such that
$$
\xi_\lambda(z,x)\leq \xi_\lambda (z,y)+\hat{\vartheta}_\lambda(z) \mbox{ for a.a. }z\in\Omega, \mbox{ all }0\leq x\leq y;
$$
  \item[$(iv)$] for every $s>0$, we can find $\mu_s>0$ such that
$$
0<\hat{\mu}_s\leq f(z,x) \mbox{ for a.a. }z\in\Omega, \mbox{ all }s\leq x
$$
$$
\mbox{ and } 0<\hat{c}\leq \liminf_{x\to0^+}\frac{f(z,x)}{x^{q_+-1}} \mbox{ uniformly for a.a. }z\in\Omega;
$$
  \item[$(v)$] for every $\rho>0$, there exists $\hat{\xi}_\rho>0$ such that for a.a. $z\in\Omega$, the function
$$
x\mapsto f(z,x)+\hat{\xi}_\rho x^{p(z)-1}
$$
is nondecreasing on $[0,\rho]$.
\end{itemize}
{\bf {\em Remarks.}} Since we are looking for positive solutions and all the above hypotheses concern the positive semiaxis $\RR_+=[0,\infty)$, without any loss of generality, we may assume that
\begin{equation}\label{eq1}
  f(z,x)=0 \mbox{ for a.a. }z\in\Omega, \mbox{ all }x\leq0.
\end{equation}

On account of hypotheses $H_1(ii),(iii)$ we see that for a.a. $z\in\Omega$, $f(z,\cdot)$ is $(p_+-1)$-superlinear. However the superlinearity property of $f(z,\cdot)$ is not expressed in terms of $AR$-condition. We recall that in the present anisotropic setting the $AR$-condition says that there exist $\vartheta>p_+$ and $M>0$ such that
\begin{equation}\label{eq2a}
   0<\vartheta F(z,x)\leq f(z,x)x \mbox{ for a.a. }z\in\Omega, \mbox{ all }x\geq M,
\end{equation}
\begin{equation}\label{eq2b}
   0<\underset{\Omega}{{\rm essinf}} F(\cdot,M).
\end{equation}

In fact this is a unilateral version of the $AR$-condition due to \eqref{eq1}. Integrating \eqref{eq2a} and using \eqref{eq2b}, we obtain the weaker condition
\begin{eqnarray*}
   &&C_0 x^\vartheta\leq F(z,x) \mbox{ for a.a. }z\in\Omega, \mbox{ all }x\geq M, \mbox{ some }C_0>0,  \\
  &\Rightarrow& C_0 x^{\vartheta-1}\leq f(z,x) \mbox{ for a.a. }z\in\Omega, \mbox{ all }x\geq M.
\end{eqnarray*}

So, the $AR$-condition implies that $f(z,\cdot)$ eventually has $(\vartheta-1)$-polynomial growth. Here we replace the $AR$-condition by the quasimonotonicity hypothesis $H_1(iii)$. This hypothesis is a slightly more general version of a condition used by Li \& Yang \cite{14Li-Yan}. Note that  there exists $M>0$ such that for a.a. $z\in\Omega$
\begin{eqnarray*}
  && x\mapsto \frac{x^{-\eta(z)}+\lambda f(z,x)}{x^{p_+-1}} \mbox{ is nondecreasing on }x\geq M  \\
  &\mbox{ or }& x\mapsto d_\lambda(z,x) \mbox{ is nondecreasing on }x\geq M.
\end{eqnarray*}
{\bf {\em Examples.}} Consider the following two functions
$$
f_1(z,x)=(x^+)^{r(z)-1} \mbox{ and } f_2(z,x)=(x^+)^{p(z)-1}\ln(1+x^+).
$$

Both functions satisfy hypothesis $H_1$, but only $f_1$ satisfies the $AR$-condition.

\medskip
By $L^{p(z)}(\Omega,\xi)$ we will denote the weighted $L^{p(z)}$-space with weight $\xi(\cdot)$. Therefore
$$
L^{p(z)}(\Omega,\xi)=\left\{u\in M(\Omega):\:\int_\Omega \xi(z)|u|^{p(z)}dz<\infty\right\}.
$$

This space is furnished with the norm
$$
\|u\|_{L^{p(z)}(\Omega,\xi)}=\inf\left\{\lambda>0:\:\int_{\Omega}\xi(z)\left|\frac{u}{\lambda}
\right|^{p(z)}dz\leq1\right\}.
$$

Note that since by hypothesis $\xi(z)>0$ for a.a. $z\in\Omega$ (see hypothesis $H_0$), the function
$$
\rho_{p(z),\xi}(u)=\int_\Omega \xi(z)|u|^{p(z)}dz,
$$
is a modular function (see Diening, Harjulehto, H\"asto \& Ruzicka \cite[Definition 2.1.1, p. 20]{7Die-Har-Has-Ruz}).

On $W^{1,p(z)}(\Omega)$ we consider the norm $\|\cdot\|$ defined earlier and a new norm given by
$$
|u|=\|Du\|_{p(z)}+\|u\|_{L^{p(z)}(\Omega,\xi)}.
$$

\begin{prop}\label{prop4}
  If hypothesis $H_0$ holds, then $\|\cdot\|$ and $|\cdot|$ are equivalent norms on $W^{1,p(z)}(\Omega)$.
\end{prop}

\begin{proof}
  From the definitions of the two norms, we have
$$
|u|\leq C_1 \|u\| \mbox{ for some } C_1>0, \mbox{ all }u\in W^{1,p(z)}(\Omega).
$$

{\it
Claim}. There exists $C_2>0$ such that $\|u\|_{p(z)}\leq C_2 |u|$ for all $u\in W^{1,p(z)}(\Omega)$.

We argue indirectly. So, suppose that the claim is not true. Then we can find $\{u_n\}_{n\geq1}\subseteq W^{1,p(z)}(\Omega)$ such that
$$
\|u_n\|_{p(z)}>n|u_n| \mbox{ for all }n\in \NN.
$$

Normalizing in $L^{p(z)}(\Omega)$, we see that we have
\begin{eqnarray}\nonumber
  && |u_n|<\frac{1}{n} \mbox{ for all }n\in\NN, \\ \nonumber
  &\Rightarrow& |u_n|\to0, \\
  &\Rightarrow& \|Du_n\|_{p(z)}\to0 \mbox{ and } \|u_n\|_{L^{p(z)}(\Omega,\xi)}\to 0\mbox{ as }n\to\infty. \label{eq3}
\end{eqnarray}

Evidently $\{u_n\}_{n\geq1}\subseteq W^{1,p(z)}(\Omega)$ is bounded. So, by passing to a suitable subsequence if necessary, we may assume that
\begin{eqnarray}\nonumber
   && u_n\overset{w}{\to}u \mbox{ in } W^{1,p(z)}(\Omega) \mbox{ and } u_n\to u\mbox{ in } L^{p(z)}(\Omega), \\
   &\Rightarrow& \|u\|_{p(z)}=1. \label{eq4}
\end{eqnarray}

From \eqref{eq3} and \eqref{eq4}, we have $u=\tilde{c}\in\RR\setminus\{0\}$ and
$$
\int_\Omega \xi(z)|u_n|^{p(z)}dz\to \int_\Omega \xi(z)|\tilde{c}|^{p(z)}dz.
$$

On account of \eqref{eq3}, we have
\begin{itemize}
  \item If $0<|\tilde{c}|\leq 1$, then $\displaystyle{0<|\tilde{c}|^{p_+}\int_\Omega \xi(z)dz\leq0}$, a contradiction.
  \item If $1<|\tilde{c}|$, then $\displaystyle{ 0\leq|\tilde{c}|^{p_-}\int_\Omega \xi(z)dz\leq0}$, a contradiction.
\end{itemize}

This proves the claim.

Using the claim and the definitions of the two norms, we conclude that $\|\cdot\|$ and $|\cdot|$ are equivalent.
\end{proof}

In what follows, we denote by $\gamma_{p(z)}:W^{1,p(z)}(\Omega)\to\RR$ the $C^1$-functional defined by
$$
\gamma_{p(z)}(u)=\int_\Omega\frac{1}{p(z)}|Du|^{p(z)}dz+\int_\Omega\frac{\xi(z)}{p(z)}|u|^{p(z)}dz \mbox{ for every }u\in W^{1,p(z)}(\Omega).
$$

For every $\lambda>0$ the energy functional $\varphi_\lambda:W^{1,p(z)}(\Omega)\to\RR$ for problem $(P_\lambda)$, is given by
$$
\varphi_\lambda(u)=\gamma_{p(z)}(u)+\int_\Omega \frac{1}{q(z)}|Du|^{q(z)}dz-\int_\Omega \frac{1}{1-\eta(z)}(u^+)^{1-\eta(z)}dz -\int_\Omega F(z,u^+)dz
$$
for all $u\in W^{1,p(z)}(\Omega)$.

On account of the third term, this functional is not $C^1$ and so the minimax theorems from the critical point theory are not directly applicable to this functional. For this reason we use truncation techniques in order to bypass the singularity and have $C^1$-functionals on which the critical point theory applies. For this reason in the next section, we deal with a purely singular problem.

Finally, we mention that, as usual, by a solution of $(P_\lambda)$, we understand a function $u\in W^{1,p(z)}(\Omega)$ such that
\begin{eqnarray*}
   && u\geq0,u\not\equiv0,u^{-\eta(\cdot)}h\in L^1(\Omega) \mbox{ for all }h\in W^{1,p(z)}(\Omega) \\
  &\mbox{and }&  \\
   && \langle A_{p(z)}(u),h\rangle +\langle A_{q(z)}(u),h\rangle+\int_\Omega \xi(z)u^{p(z)-1}hdz \\
   &=& \int_\Omega \left[u^{-\eta(z)}+\lambda f(z,u)\right]hdz \mbox{ for all }h\in W^{1,p(z)}(\Omega).
\end{eqnarray*}

\section{A purely singular problem}

In this section we deal with the following purely singular problem
\begin{equation}\label{eq5}
  \left\{
\begin{array}{lll}
-\Delta_{p(z)}u(z)-\Delta_{q(z)}u(z)+\xi(z)u(z)^{p(z)-1}=u(z)^{-\eta(z)} \text{ in } \Omega,\\
\di\frac{\partial u}{\partial n}=0 \mbox{ on }\partial\Omega,\ u>0.
\end{array}
\right.
\end{equation}


To solve \eqref{eq5}, we first consider a perturbation of \eqref{eq5} which removes the singularity. So, we consider the following approximation of problem \eqref{eq5}:
\begin{equation}\tag{$8_\varepsilon$}\label{eq8ep}
  \left\{
\begin{array}{lll}
-\Delta_{p(z)}u(z)-\Delta_{q(z)}u(z)+\xi(z)u(z)^{p(z)-1}=[u(z)+\varepsilon]^{-\eta(z)} \text{ in } \Omega,\\
\di\frac{\partial u}{\partial n}=0 \mbox{ on }\partial\Omega,\ u>0.
\end{array}
\right.
\end{equation}

We solve this problem using a topological approach (fixed point theory). So, given $g\in L^{p(z)}(\Omega)$, $g\geq0$ and $\varepsilon\in(0,1]$, we consider the following problem:
\begin{equation}\label{eq6}
  \left\{
\begin{array}{lll}
-\Delta_{p(z)}u(z)-\Delta_{q(z)}u(z)+\xi(z)u(z)^{p(z)-1}=[g(z)+\varepsilon]^{-\eta(z)} \text{ in } \Omega,\\
\di\frac{\partial u}{\partial n}=0 \mbox{ on }\partial\Omega,\ u>0.
\end{array}
\right.
\end{equation}

For this problem we have the following result.

\begin{prop}\label{prop5}
  If hypotheses $H_0$ hold, problem \eqref{eq6} admits a unique solution $\tilde{u}_\varepsilon\in{\rm int}\,C_+$.
\end{prop}

\begin{proof}
  Let $K_{p(z)}: L^{p(z)}(\Omega)\to L^{p'(z)}(\Omega)$ be the map defined by
$$
K_{p(z)}(u)=|u|^{p(z)-2}u \mbox{ for all }u\in L^{p(z)}(\Omega).
$$

Evidently this map is bounded, continuous, strictly monotone. Then we consider the operator $V:W^{1,p(z)}(\Omega)\to W^{1,p(z)}(\Omega)^* $ defined by
$$
V(u)=A_{p(z)}(u)+A_{q(z)}(u) +\xi(z)K_{p(z)}(u) \mbox{ for all }u\in W^{1,p(z)}(\Omega).
$$

This operator is bounded, continuous, strictly monotone (thus, maximal monotone too). Also, if $u\in W^{1,p(z)}(\Omega)$, we have
\begin{eqnarray*}
  \langle V(u),u\rangle &=& \rho_{p(z)}(Du)+\rho_{q(z)}(Du)+\int_\Omega \xi(z)|u|^{p(z)}dz \\
   &\geq& \rho_{p(z)}(Du)+\int_\Omega \xi(z)|u|^{p(z)}dz \\
   &\geq& \|Du\|_{p(z)}+\|u\|_{L^{p(z)}(\Omega,\xi)}-1 \\
   && \left(\mbox{see Corollary 2.1.15 of \cite[p. 25]{7Die-Har-Has-Ruz} and recall that $\rho_{p(z),\xi}(\cdot)$ is modular}\right) \\
   &\geq& C_2\|u\| \mbox{ for some } C_2>0 \mbox{ (see Proposition \ref{prop4}), }   \\
   &\Rightarrow& V(\cdot) \mbox{ is coercive. }
\end{eqnarray*}

We know that a maximal monotone coercive operator is surjective (see Papageorgiou, R\u adulescu \& Repov\v s \cite[Corollary 2.8.7, p. 135]{15Pap-Rad-Rep}). Since $[g(\cdot)+\varepsilon]^{-\eta(\cdot)}\in L^\infty(\Omega)$, we can find $\tilde{u}_\varepsilon\in W^{1,p(z)}(\Omega)$ such that
$$
V(\tilde{u}_\varepsilon)=[g(\cdot)+\varepsilon]^{-\eta(\cdot)}.
$$

Moreover, the strict monotonicity of $V(\cdot)$ implies that this solution is unique. Proposition 3.1 of Gasinski \& Papageorgiou \cite{9Gas-Pap} implies that $\tilde{u}_\varepsilon\in L^{\infty}(\Omega)$. Then by Theorem 1.3 of Fan \cite{8Fan} (see also Lieberman \cite{13Lieberman}), we have that $\tilde{u}_\varepsilon\in C_+\setminus\{0\}$. Finally the anisotropic maximum principle of Zhang \cite{23Zha} implies that $\tilde{u}_\varepsilon\in{\rm int}\,C_+$.
\end{proof}

Using Proposition \ref{prop5} we can define the solution map $L_\varepsilon:L^{p(z)}(\Omega)\to L^{p(z)}(\Omega)$ for problem \eqref{eq6} by
\begin{equation}\label{eq7}
  L_\varepsilon(g)=\tilde{u}_\varepsilon.
\end{equation}

Clearly, a fixed point of this map will be a solution for problem \eqref{eq8ep}.

\begin{prop}\label{prop6}
  If hypotheses $H_0$ hold, then problem \eqref{eq8ep} admits a unique solution $\overline{u}_\varepsilon\in {\rm int}\,C_+$.
\end{prop}

\begin{proof}
  From Proposition \ref{prop5}, we have
\begin{equation}\label{eq8}
  \langle A_{p(z)}(\tilde{u}_\varepsilon),h \rangle+\langle A_{q(z)}(\tilde{u}_\varepsilon),h\rangle+\int_\Omega \xi(z) \tilde{u}_\varepsilon^{p(z)-1}hdz=\int_\Omega [g(z)+\varepsilon]^{-\eta(z)}hdz
\end{equation}
for all $h\in W^{1,p(z)}(\Omega)$.

In \eqref{eq8} we choose $h=\tilde{u}_\varepsilon=L_\varepsilon(g)\in W^{1,p(z)}(\Omega)$ (see \eqref{eq7}) and we obtain
\begin{eqnarray}\nonumber
  &&\rho_{p(z)}(D\tilde{u}_\varepsilon)+\rho_{q(z)}(D\tilde{u}_\varepsilon)+\int_\Omega \xi(z)\tilde{u}_\varepsilon^{p(z)}dz=\int_\Omega \frac{\tilde{u}_\varepsilon}{[g(z)+\varepsilon]^{\eta(z)}}dz, \\
  &\Rightarrow& \left\{\begin{array}{ll} &  { \|L_\varepsilon(g)\|^{p_+}\leq C_3\|L_\varepsilon(g)\|\ \mbox{if}\   \|L_\varepsilon(g)\|\leq 1 } \\
  &  { \|L_\varepsilon(g)\|^{p_-}\leq C_3\|L_\varepsilon(g)\|\ \mbox{if}\   \|L_\varepsilon(g)\|> 1\ \mbox{for some } C_3=C_3(\varepsilon)>0, \mbox{ all }g\in L^\infty(\Omega) }
  \end{array}\right. \label{eq9}\\ \nonumber
  && \mbox{ (see Proposition \ref{prop4}). }
\end{eqnarray}

Next, we show that $L_\varepsilon(\cdot)$ is continuous. To this end let $g_n\to g$ in $L^{p(z)}(\Omega)$. From \eqref{eq9} we have that
$$
\left\{L_\varepsilon(g_n)=\tilde{u}_{\varepsilon_n}=\tilde{u}_n\right\}_{n\geq1}\subseteq W^{1,p(z)}(\Omega) \mbox{ is bounded. }
$$

So, we may assume that
\begin{equation}\label{eq10}
  \tilde{u}_n\overset{w}{\to}\tilde{u} \mbox{ in } W^{1,p(z)}(\Omega) \mbox{ and } \tilde{u}_n\to \tilde{u} \mbox{ in }L^{p(z)}(\Omega).
\end{equation}

We have
\begin{equation}\label{eq11}
  \langle A_{p(z)}(\tilde{u}_n),h\rangle+\langle A_{q(z)}(\tilde{u}_n),h\rangle+\int_\Omega \xi(z)(\tilde{u}_n)^{p(z)-1}hdz=\int_\Omega \frac{h}{[g_n+\varepsilon]^{\eta(z)}}dz
\end{equation}
for all $h\in W^{1,p(z)}(\Omega)$, all $n\in \NN$.

In \eqref{eq11} we choose $h=\tilde{u}_n-\tilde{u}\in W^{1,p(z)}(\Omega)$, pass to the limit as $n\to\infty$ and use \eqref{eq10}. Then we have
\begin{eqnarray}\nonumber
   && \lim_{n\to\infty}\left[\langle A_{p(z)}(\tilde{u}_n),\tilde{u}_n-\tilde{u}\rangle+\langle A_{q(z)}(\tilde{u}_n),\tilde{u}_n-\tilde{u}\rangle\right]=0, \\ \nonumber
  &\Rightarrow& \limsup_{n\to\infty}\left[\langle A_{p(z)}(\tilde{u}_n),\tilde{u}_n-\tilde{u}\rangle+\langle A_{q(z)}(\tilde{u}),\tilde{u}_n-\tilde{u}\rangle\right]\leq0 \\ \nonumber
   && \mbox{ (since $A_{q(z)}(\cdot)$ is monotone), } \\ \nonumber
   &\Rightarrow& \limsup_{n\to\infty}\langle A_{p(z)}(\tilde{u}_n),\tilde{u}_n-\tilde{u}\rangle\leq 0 \mbox{ (see \eqref{eq10}), } \\
   &\Rightarrow& \tilde{u}_n\to \tilde{u} \mbox{ in } W^{1,p(z)}(\Omega) \mbox{ (see Proposition \eqref{prop2}).} \label{eq12}
\end{eqnarray}

If in \eqref{eq11} we pass to the limit as $n\to\infty$ and use \eqref{eq12}, we obtain that
\begin{eqnarray*}
  && \langle A_{p(z)}(\tilde{u}),h\rangle+\langle A_{q(z)}(\tilde{u}),h\rangle+\int_\Omega \xi(z)\tilde{u}^{p(z)-1}hdz=\int_\Omega \frac{h}{[g+\varepsilon]^{\eta(z)}}dz \\
   && \mbox{ for all }h\in W^{1,p(z)}(\Omega), \\
   &\Rightarrow& \tilde{u}=L_\varepsilon(g), \\
   &\Rightarrow& L_\varepsilon(\cdot) \mbox{ is continuous. }
\end{eqnarray*}

The continuity of $L_\varepsilon(\cdot)$ together with \eqref{eq9} and the compact embedding of $W^{1,p(z)}(\Omega)$ into $L^{p(z)}(\Omega)$, permit the use of Schauder-Tychonov fixed point theorem (see, for example, Papageorgiou \& Winkert \cite[Theorem 6.8.5, p. 581]{20Pap-Win}) and we find $\overline{u}_\varepsilon\in{\rm int}\,C_+$ such that
\begin{eqnarray*}
  && L_\varepsilon(\overline{u}_\varepsilon)=\overline{u}_\varepsilon, \\
  &\Rightarrow& \overline{u}_\varepsilon\in{\rm int}\,C_+ \mbox{ is a positive solution of }\eqref{eq5}.
\end{eqnarray*}

Next we show the uniqueness of this solution. Suppose that $\hat{u}_\varepsilon\in W^{1,p(z)}(\Omega)$ is another positive solution of \eqref{eq5}. Again we have $\hat{u}_\varepsilon\in {\rm int}\,C_+$. Also, we have
\begin{eqnarray*}
  0 &\leq& \langle A_{p(z)}(\overline{u}_\varepsilon)-A_{p(z)}(\hat{u}_\varepsilon),(\overline{u}_\varepsilon-\hat{u}_\varepsilon)^+\rangle+
\langle A_{q(z)}(\overline{u}_\varepsilon)-A_{q(z)}(\hat{u}_\varepsilon),(\overline{u}_\varepsilon-\hat{u}_\varepsilon)^+\rangle\\
&+&\int_\Omega\xi(z)(\overline{u}_\varepsilon^{p(z)-1}-\hat{u}_\varepsilon^{p(z)-1})(\overline{u}_\varepsilon-
\hat{u}_\varepsilon)^+ dz= \\
  &=& \int_\Omega \left[\frac{1}{[\overline{u}_\varepsilon+\varepsilon]^{\eta(z)}}-\frac{1}{[\hat{u}_\varepsilon+\varepsilon]^{\eta(z)}}
\right](\overline{u}_\varepsilon-\hat{u}_\varepsilon)^+dz\leq 0 \\
  \Rightarrow \overline{u}_\varepsilon&\leq& \hat{u}_\varepsilon.
\end{eqnarray*}

Interchanging the roles of $\overline{u}_\varepsilon$ and $\hat{u}_\varepsilon$ in the above argument, we also have that $\hat{u}_\varepsilon\leq \overline{u}_\varepsilon$, therefore $\overline{u}_\varepsilon=\hat{u}_\varepsilon$. This proves the uniqueness of the positive solution $\overline{u}_\varepsilon\in{\rm int}\,C_+$ of problem \eqref{eq8ep}.
\end{proof}

Evidently, to produce a positive solution of \eqref{eq5}, we will let $\varepsilon\to0^+$.

To this end, the following monotonicity property of the map $\varepsilon\mapsto\overline{u}_\varepsilon$ will be useful.

\begin{prop}\label{prop7}
  If hypotheses $H_0$ hold, then the map $\varepsilon\mapsto\overline{u}_\varepsilon$ from $(0,1]$ into $C_+$ is nonincreasing, that is,
$$
0<\varepsilon'<\varepsilon\leq 1\Rightarrow \overline{u}_\varepsilon\leq \overline{u}_{\varepsilon'}.
$$
\end{prop}

\begin{proof}
  Let $0<\varepsilon'<\varepsilon\leq 1$ and consider $\overline{u}_{\varepsilon'}$, $\overline{u}_\varepsilon\in{\rm int}\,C_+$ the corresponding unique positive solutions of problems $(8_{\varepsilon'})$ and  \eqref{eq8ep} respectively, established in Proposition \ref{prop6}.

We have
\begin{eqnarray}\nonumber
   && -\Delta_{p(z)}\overline{u}_{\varepsilon'}-\Delta_{q(z)}\overline{u}_{\varepsilon'}+\xi(z)
\overline{u}_{\varepsilon'}^{p(z)-1} \\ \nonumber
   &=& [\overline{u}_{\varepsilon'}+\varepsilon']^{-\eta(z)} \\
   &\geq& [\overline{u}_{\varepsilon'}+\varepsilon]^{-\eta(z)} \mbox{  in }\Omega \mbox{ (since $0<\varepsilon'<\varepsilon$). } \label{eq13}
\end{eqnarray}

We introduce the Carath\' eodory function $e_\varepsilon(z,x)$ defined by
\begin{equation}\label{eq14}
  e_\varepsilon(z,x)=\left\{
                       \begin{array}{ll}
                     \di    \frac{1}{[x^++\varepsilon]^{\eta(z)}}, & \hbox{ if } x\leq\overline{u}_{\varepsilon'}(z) \\
                      \di   \frac{1}{[\overline{u}_{\varepsilon'}(z)+\varepsilon]^{\eta(z)}}, & \hbox{ if } \overline{u}_{\varepsilon'}(z)<x.
                       \end{array}
                     \right.
\end{equation}

We set $E_\varepsilon(z,x)=\displaystyle{\int_0^x e_\varepsilon(z,s)ds}$ and consider the $C^1$-functional $\tau_\varepsilon: W^{1,p(z)}(\Omega)\to\RR$ defined by
$$
\tau_\varepsilon(u)=\gamma_{p(z)}(u)+\int_\Omega\frac{1}{q(z)}|Du|^{q(z)}dz-\int_\Omega E_\varepsilon(z,u)dz \mbox{  for all }u\in W^{1,p(z)}(\Omega).
$$

We have
\begin{eqnarray*}
  \tau_\varepsilon(u) &\geq& \rho_{p(z)}(Du)+\rho_{p(z),\xi}(u)-C_4 \mbox{ for some }C_4>0 \mbox{ (see \eqref{eq14}) } \\
  &\geq& |u|-C_4-1, \\
  \Rightarrow \tau_\varepsilon(\cdot) &\mbox{is}& \mbox{coercive (see Proposition \ref{prop4}). }
\end{eqnarray*}

Also, exploiting the compact embedding of $W^{1,p(z)}(\Omega)$ into $L^{p(z)}(\Omega)$, we have that
$$
\tau_\varepsilon(\cdot) \mbox{ is sequentially weakly lower semicontinuous. }
$$

Invoking the Weierstrass-Tonelli theorem, we can find $\hat{u}_\varepsilon\in W^{1,p(z)}(\Omega)$ such that
\begin{eqnarray*}
  && \tau_\varepsilon(\hat{u}_\varepsilon)=\min\left\{\tau_\varepsilon(u):\:u\in W^{1,p(z)}(\Omega)\right\}, \\
  &\Rightarrow& \tau'_\varepsilon(\hat{u}_\varepsilon)=0,
\end{eqnarray*}
\begin{equation}\label{eq15}
  \Rightarrow \langle A_{p(z)}(\hat{u}_\varepsilon),h\rangle+\langle A_{q(z)}(\hat{u}_\varepsilon),h\rangle+\int_\Omega \xi(z)|\hat{u}_\varepsilon|^{p(z)-2}\hat{u}_\varepsilon h dz=\int_\Omega e_\varepsilon(z,\hat{u}_\varepsilon)h dz
\end{equation}
for all $h\in W^{1,p(z)}(\Omega)$.

In \eqref{eq15} we choose $h=-\hat{u}_\varepsilon^-\in W^{1,p(z)}(\Omega)$ and obtain
\begin{eqnarray*}
  && \rho_{p(z)}(D\hat{u}_\varepsilon^-)+\rho_{q(z)}(D\hat{u}_\varepsilon^-)+\rho_{p(z),\xi}(\hat{u}_\varepsilon^-)=0 \mbox{ (see \eqref{eq14}),} \\
  &\Rightarrow& \rho_{p(z)}(Du_\varepsilon^-)+\rho_{p(z),\xi}(\hat{u}_\varepsilon^-)\leq0 \\
  &\Rightarrow& \hat{u}_\varepsilon\geq0,\ \hat{u}_\varepsilon\not=0 \mbox{ (see \eqref{eq15}). }
\end{eqnarray*}

Next, in \eqref{eq15} we choose $h=[\hat{u}_\varepsilon-\overline{u}_{\varepsilon'}]^+\in W^{1,p(z)}(\Omega)$. We obtain
\begin{eqnarray*}
  && \langle A_{p(z)}(\hat{u}_\varepsilon),(\hat{u}_\varepsilon-\overline{u}_{\varepsilon'})^+\rangle+\langle A_{q(z)}(\hat{u_\varepsilon}), (\hat{u}_\varepsilon-\overline{u}_{\varepsilon'})^+\rangle+\int_\Omega \xi(z)\hat{u}_\varepsilon^{p(z)-1}(\hat{u}_\varepsilon-\overline{u}_{\varepsilon'})^+dz \\
   &=& \int_\Omega \frac{[\hat{u}_\varepsilon-\overline{u}_{\varepsilon'}]^+}{[\overline{u}_{\varepsilon'}+\varepsilon]^{\eta(z)}}dz \mbox{ (see \eqref{eq14}) } \\
  &\leq& \langle A_{p(z)}(\overline{u}_{\varepsilon'}), (\hat{u}_\varepsilon-\overline{u}_{\varepsilon'})^+\rangle+
\langle A_{q(z)}(\overline{u}_{\varepsilon'}),(\hat{u}_\varepsilon-\overline{u}_{\varepsilon'})^+\rangle+\int_\Omega \xi(z)\overline{u}_{\varepsilon'}^{p(z)-1}(\hat{u}_\varepsilon-\overline{u}_{\varepsilon'})^+dz, \\
  \Rightarrow \hat{u}_\varepsilon&\leq& \overline{u}_{\varepsilon'}.
\end{eqnarray*}

So, we have proved that
\begin{eqnarray*}
  && \hat{u}_\varepsilon\in[0,\overline{u}_{\varepsilon'}],\ \hat{u}_\varepsilon\not=0, \\
  &\Rightarrow& \hat{u}_\varepsilon=\overline{u}_\varepsilon \mbox{ (see \eqref{eq15}, \eqref{eq14} and Proposition \ref{prop6}),  } \\
  &\Rightarrow& \overline{u}_\varepsilon\leq \overline{u}_{\varepsilon'}.
\end{eqnarray*}
The proof is now complete.
\end{proof}

Now we will pass to the limit as $\varepsilon\to0^+$ and produce a solution for problem \eqref{eq5}.

\begin{prop}\label{prop8}
  If hypotheses $H_0$ hold, then problem \eqref{eq5} admits a unique solution $\overline{u}\in {\rm int}\,C_+$.
\end{prop}

\begin{proof}
  Let $\{\varepsilon_n\}_{n\geq1}\subseteq(0,1]$ be such that $\varepsilon_n\to0^+$ and let $\overline{u}_n=\overline{u}_{\varepsilon_n}\in{\rm int}\,C_+$ be as in Proposition \ref{prop6}. We have
\begin{equation}\label{eq16}
  \langle A_{p(z)}(\overline{u}_n),h\rangle+\langle A_{q(z)}(\overline{u}_n),h\rangle+\int_\Omega \xi(z)\overline{u}_n^{p(z)-1}hdz=\int_\Omega \frac{h}{[\overline{u}_n+\varepsilon_n]^{\eta(z)}}dz
\end{equation}
for all $h\in W^{1,p(z)}(\Omega)$, all $n\in \NN$.

We choose $h=\overline{u}_n\in W^{1,p(z)}(\Omega)$. We obtain
\begin{eqnarray*}
   && \rho_{p(z)}(D\overline{u}_n)+\rho_{q(z)}(D\overline{u}_n)+\int_\Omega \xi(z)\overline{u}_n^{p(z)}dz\leq \int_\Omega \frac{\overline{u}_n}{\overline{u}_1^{\eta(z)}}\\
   && \mbox{(since $\overline{u}_1\leq\overline{u}_n$ for all $n\in\NN$, see Proposition \ref{prop7})} \\
  &\Rightarrow& \{\overline{u}_n\}_{n\geq1}\subseteq W^{1,p(z)}(\Omega) \mbox{ is bounded } \\
  && \mbox{ (see Proposition \ref{prop4} and recall that $p_->1$). }
\end{eqnarray*}

Proposition 3.1 of Gasinski \& Papageorgiou \cite{9Gas-Pap}, implies that we can find $C_5>0$ such that
$$
\|\overline{u}_n\|_\infty\leq C_5 \mbox{ for all }n\in\NN.
$$

Then the anisotropic regularity theory of Fan \cite[Theorem 1.3]{8Fan}  implies that we can find $\alpha\in(0,1)$ and $C_6>0$ such that
$$
\overline{u}_n\in C^{1,\alpha}(\overline{\Omega}), \ \|\overline{u}_n\|_{C^{1,\alpha}(\overline{\Omega})}\leq C_6 \mbox{ for all }n\in\NN.
$$

The compact embedding of $C^{1,\alpha}(\overline{\Omega})$ into $C^1(\overline{\Omega})$ and the monotonicity of the sequence $\{\overline{u}_n\}_{n\geq1}$ (see Proposition \ref{prop7}), imply that
\begin{equation}\label{eq17}
  \overline{u}_n\to\overline{u} \mbox{ in } C^1(\overline{\Omega}).
\end{equation}

Since $\overline{u}_1\leq \overline{u}_n$ for all $n\in \NN$, we have $\overline{u}\not=0$ and so $\overline{u}\in{\rm int}\,C_+$. Moreover, passing to the limit as $n\to\infty$ in \eqref{eq16} and using \eqref{eq17}, we conclude that $\overline{u}\in{\rm int}\,C_+$ is a positive solution of problem \eqref{eq5}.

Finally, we show the uniqueness of this positive solution. So, suppose that $\tilde{u}\in W^{1,p(z)}(\Omega)$ is another positive solution of \eqref{eq5}. As in the proof of Proposition \ref{prop6}, using the fact that the map $x\mapsto x^{-\eta(z)}$ is strictly decreasing on $(0,+\infty)$, we obtain
\begin{eqnarray*}
  && \tilde{u}=\overline{u}, \\
  &\Rightarrow& \overline{u}\in{\rm int}\,C_+ \mbox{ is the unique positive solution of \eqref{eq5}.}
\end{eqnarray*}
The proof is now complete.
\end{proof}

In the next section we will use $\overline{u}\in{\rm int}\,C_+$ and truncation techniques to bypass the singularity and show that problem $(P_\lambda)$ has positive solutions for certain values of the parameter $\lambda>0$.

\section{Positive solutions}

We introduce the following two sets
\begin{eqnarray*}
  && \mathcal{L}=\{\lambda>0:\:\mbox{ problem $(P_\lambda)$ has a positive solution}\}, \\
  && S_\lambda=\mbox{set of positive solutions of problem $(P_\lambda)$. }
\end{eqnarray*}

We start by showing the nonemptiness of $\mathcal{L}$ (=the set of admissible parameter values).

\begin{prop}\label{prop9}
  If hypotheses $H_0$, $H_1$ hold, then $\mathcal{L}\not=\emptyset$.
\end{prop}

\begin{proof}
  Let $\overline{u}\in{\rm int}\,C_+$ be the unique positive solution of problem \eqref{eq5} produced in Proposition \ref{prop8}.

We consider the following auxiliary problem:
\begin{equation}\label{eq18}
  \left\{
\begin{array}{lll}
-\Delta_{p(z)}u(z)-\Delta_{q(z)}u(z)+\xi(z)u(z)^{p(z)-1}=\overline{u}(z)^{-\eta(z)}+1 \text{ in } \Omega,\\
\di \frac{\partial u}{\partial n}=0 \mbox{ on }\partial\Omega,\ u>0.
\end{array}
\right.
\end{equation}

From Proposition \ref{prop5}, we know that this problem has a unique positive solution $\hat{u}\in{\rm int}\,C_+$. We have
\begin{eqnarray}\nonumber
  && \langle A_{p(z)}(\overline{u}),(\overline{u}-\hat{u})^+\rangle+\langle A_{q(z)}(\overline{u}),(\overline{u}-\hat{u})^+\rangle+\int_\Omega \xi(z)\overline{u}^{p(z)-1}(\overline{u}-\hat{u})^+dz \\ \nonumber
   &=& \int_\Omega \overline{u}^{-\eta(z)}(\overline{u}-\hat{u})^+dz \mbox{ (see \eqref{eq5}) } \\ \nonumber
   &\leq& \int_\Omega \left[\overline{u}^{-\eta(z)}+1\right](\overline{u}-\hat{u})^+dz \\ \nonumber
   &=& \langle A_{p(z)}(\hat{u}),(\overline{u}-\hat{u})^+\rangle+\langle A_{q(z)}(\hat{u}),(\overline{u}-\hat{u})^+\rangle+\int_\Omega \xi(z)\hat{u}^{p(z)-1}(\overline{u}-\hat{u})^+dz \\ \nonumber
  && \mbox{ (since $\hat{u}\in{\rm int}\,C_+$ solves \eqref{eq18}), } \\
  \Rightarrow \overline{u}&\leq&\hat{u}. \label{eq19}
\end{eqnarray}

Since $\hat{u}\in{\rm int}\,C_+$, on account of hypothesis $H_1(i)$, we have
$$
0\leq f(\cdot,\hat{u}(\cdot))\in L^{\infty}(\Omega).
$$

So, we can find $\lambda_0>0$ such that
\begin{equation}\label{eq20}
  0\leq \lambda f(z,\hat{u}(z))\leq1 \mbox{ for a.a. }z\in\Omega, \mbox{ all }\lambda\in(0,\lambda_0].
\end{equation}

We introduce the Carath\' eodory function $\beta_\lambda(z,x)$ defined by
\begin{equation}\label{eq21}
  \beta_\lambda(z,x)=\left\{
                       \begin{array}{ll}
                         \overline{u}(z)^{-\eta(z)}+\lambda f(z,\overline{u}(z)), & \hbox{ if } x<\overline{u}(z) \\
                         x^{-\eta(z)}+\lambda f(z,x), & \hbox{ if }\overline{u}(z)\leq x\leq \hat{u}(z) \mbox{ (see \eqref{eq19}) } \\
                         \hat{u}(z)^{-\eta(z)}+\lambda f(z,\hat{u}(z)), & \hbox{ if }\hat{u}(z)<x.
                       \end{array}
                     \right.
\end{equation}

We set $B_\lambda(z,x)=\displaystyle{\int_0^x}\beta_\lambda(z,s)ds$ and consider the $C^1$-functional $\tilde{\Psi}_\lambda: W^{1,p(z)}(\Omega)\to \RR$ defined by
$$
\tilde{\Psi}_\lambda(u)=\gamma_{p(z)}(u)+\int_\Omega \frac{1}{q(z)}|Du|^{q(z)}dz-\int_\Omega B_\lambda(z,u) dz \mbox{ for all }u\in W^{1,p(z)}(\Omega), \mbox{ all }\lambda\in(0,\lambda_0].
$$

From \eqref{eq19} it is clear that $\tilde{\Psi}_\lambda(\cdot)$ is coercive. Also, it is sequentially weakly lower semicontinuous. Hence, we can find $\tilde{u}\in W^{1,p(z)}(\Omega)$ such that
\begin{eqnarray*}
  \tilde{\Psi}_\lambda(\tilde{u}) &=& \min\left\{\tilde{\Psi}_\lambda(u):\:u\in W^{1,p(z)}(\Omega)\right\}, \\
  \Rightarrow \tilde{\Psi}'_\lambda(\tilde{u}) &=& 0,
\end{eqnarray*}
\begin{equation}\label{eq22}
  \Rightarrow \langle A_{p(z)}(\tilde{u}),h\rangle+\langle A_{q(z)}(\tilde{u}),h\rangle+\int_\Omega \xi(z)|\tilde{u}|^{p(z)-2}\tilde{u}hdz=\int_\Omega \beta_\lambda(z,\tilde{u})hdz
\end{equation}
for all $h\in W^{1,p(z)}(\Omega)$.

In \eqref{eq22} first we choose $h=(\overline{u}-\tilde{u})^+\in W^{1,p(z)}(\Omega)$. We have
\begin{eqnarray*}
  && \langle A_{p(z)}(\tilde{u}),(\overline{u}-\tilde{u})^+\rangle+\langle A_{q(z)}(\tilde{u}),(\overline{u}-\tilde{u})^+\rangle+\int_\Omega \xi(z)|\tilde{u}|^{p(z)-2}\tilde{u}(\overline{u}-\tilde{u})^+dz \\
   &=& \int_\Omega \left[\overline{u}^{-\eta(z)}+\lambda f(z,\overline{u})\right](\overline{u}-\tilde{u})^+ dz \mbox{ (see \eqref{eq21}) } \\
   &\geq& \int_\Omega \overline{u}^{-\eta(z)}(\overline{u}-\tilde{u})^+dz \mbox{ (see hypothesis $H_1(i)$) } \\
   &=& \langle A_{p(z)}(\overline{u}),(\overline{u}-\tilde{u})^+\rangle+\langle A_{q(z)},(\overline{u}-\tilde{u})^+\rangle+\int_\Omega \xi(z) \overline{u}^{p(z)-1}(\overline{u}-\tilde{u})^+dz, \\
   \Rightarrow \overline{u}&\leq & \tilde{u}.
\end{eqnarray*}

Next, in \eqref{eq22} we choose $h=(\tilde{u}-\hat{u})^+\in W^{1,p(z)}(\Omega)$. We have
\begin{eqnarray*}
  && \langle A_{p(z)}(\tilde{u}), (\tilde{u}-\hat{u})^+\rangle+\langle A_{q(z)},(\tilde{u}-\hat{u})^+\rangle+\int_\Omega \xi(z)\tilde{u}^{p(z)-1}(\tilde{u}-\hat{u})^+dz \\
   &=& \int_\Omega\left[\hat{u}^{-\eta(z)}+\lambda f(z,\hat{u})\right](\tilde{u}-\hat{u})^+dz \mbox{(see \eqref{eq21})} \\
   &\leq& \int_\Omega \left[\hat{u}^{-\eta(z)}+1\right](\tilde{u}-\hat{u})^+dz \mbox{ (since $\lambda\in(0,\lambda_0]$, see \eqref{eq20}) } \\
   &\leq& \int_\Omega\left[\overline{u}^{-\eta(z)}+1\right](\tilde{u}-\hat{u})^+dz \mbox{ (see \eqref{eq19}) } \\
   &=& \langle A_{p(z)}(\hat{u}),(\tilde{u}-\hat{u})^+\rangle+\langle A_{q(z)}(\hat{u}),(\tilde{u}-\hat{u})^+\rangle+\int_\Omega \xi(z)\tilde{u}^{p(z)-1}(\tilde{u}-\hat{u})^+dz, \\
  \Rightarrow \tilde{u} &\leq& \hat{u}.
\end{eqnarray*}

So, we have proved that
\begin{equation}\label{eq23}
  \tilde{u}\in[\overline{u},\hat{u}].
\end{equation}

From \eqref{eq23}, \eqref{eq21} and \eqref{eq22}, we infer that
\begin{eqnarray*}
  && \tilde{u}\in S_\lambda, \\
  &\Rightarrow& (0,\lambda_0]\subseteq\mathcal{L}\not=\emptyset,
\end{eqnarray*}
which concludes the proof.
\end{proof}

\begin{prop}\label{prop10}
  If hypotheses $H_0$, $H_1$ hold and $\lambda\in\mathcal{L}$, then $\overline{u}\leq u$ for all $u\in S_\lambda$.
\end{prop}

\begin{proof}
  Let $u\in S_\lambda$ and consider the following function
\begin{equation}\label{eq24}
  \hat{\mu}_+(z,x)=\left\{
                     \begin{array}{ll}
                       x^{-\eta(z)}, & \hbox{ if }0\leq x\leq u(z) \\
                       u(z)^{-\eta(z)}, & \hbox{ if }u(z)<x.
                     \end{array}
                   \right.
\end{equation}

Evidently, this function is Carath\'eodory on $\Omega\times(\RR\setminus\{0\})$ and is singular at $x=0$. We consider the purely singular problem
\begin{equation}\label{eq25}
  \left\{
\begin{array}{lll}
-\Delta_{p(z)}u(z)-\Delta_{q(z)}u(z)+\xi(z)u(z)^{p(z)-1}=\hat{\mu}_+(z,u(z)) \text{ in } \Omega,\\
\di\frac{\partial u}{\partial n}=0 \mbox{ on }\partial\Omega,\ u>0.
\end{array}
\right.
\end{equation}

Reasoning as for problem \eqref{eq5}, we show that problem \eqref{eq25} admits a positive solution $u^*\in{\rm int}\,C_+$ (see also Papageorgiou, R\u adulescu \& Repov\v s \cite[Proposition 10]{16Pap-Rad-Rep}). We have
\begin{eqnarray*}
  && \langle A_{p(z)}(u^*),(u^*-u)^+\rangle+\langle A_{q(z)}(u^*),(u^*-u)^+\rangle+\int_\Omega \xi(z)(u^*)^{p(z)-1}(u^*-u)^+dz \\
  &=& \int_\Omega u^{-\eta(z)}(u^*-u)^+dz \mbox{ (see \eqref{eq24}) } \\
  && \mbox{(recall that, by definition, $u^{-\eta(z)}h\in L^1(\Omega)$ for all $h\in W^{1,p(z)}(\Omega)$)} \\
  &\leq& \int_\Omega \left[u^{-\eta(z)}+\lambda f(z,u)\right](u^*-u)^+dz \\
  &=& \langle A_{p(z)}(u),(u^*-u)^+\rangle+\langle A_{q(z)}(u),(u^*-u)^+\rangle+\int_\Omega \xi(z) u^{p(z)-1}(u^*-u)^+dz \mbox{ (since $u\in S_\lambda$), }\\
  \Rightarrow u^*&\leq& u.
\end{eqnarray*}

So, we have
\begin{equation}\label{eq26}
  u^*\in[0,u],\;u^*\not=0.
\end{equation}

From \eqref{eq26}, \eqref{eq24} and Proposition \ref{prop8}, it follows that
\begin{eqnarray*}
  && u^*=\overline{u} \\
  &\Rightarrow& \overline{u}\leq u \mbox{ for all }u\in S_\lambda.
\end{eqnarray*}
The proof is now complete.
\end{proof}

According to the previous proposition, if $\lambda\in \mathcal{L}$ and $u\in S_\lambda$, then
$$
0\leq u^{-\eta(z)}\leq \overline{u}^{-\eta(z)} \mbox{ with } \overline{u}(\cdot)^{-\eta(\cdot)}\in L^\infty(\Omega).
$$

Then the anisotropic regularity theory (see Fan \cite{8Fan}) and the anisotropic maximum principle (see Zhang \cite{23Zha}), imply the following result concerning the solution set $S_\lambda$.

\begin{prop}\label{prop11}
  If hypotheses $H_0$, $H_1$ hold and $\lambda\in\mathcal{L}$, then $S_\lambda\subseteq{\rm int}\,C_+$.
\end{prop}

Let $\lambda^*=\sup \mathcal{L}$.

\begin{prop}\label{prop12}
  If hypotheses $H_0$, $H_1$ hold, then $\lambda^*<\infty$.
\end{prop}

\begin{proof}
  On account of hypotheses $H_1(ii),(iii),(iv)$, we can find $\lambda_0>0$ big such that
\begin{equation}\label{eq27}
  \lambda_0f(z,x)\geq\xi(z)x^{p(z)-1} \mbox{ for a.a. }z\in\Omega,\mbox{ all }x\geq0.
\end{equation}

Let $\lambda>\lambda_0$ and suppose that $\lambda\in \mathcal{L}$. Then we can find $u_\lambda\in S_\lambda\subseteq{\rm int}\,C_+$ (see Proposition \eqref{prop11}). Let $m_\lambda=\displaystyle{\min_{\overline{\Omega}}u_\lambda>0}$ (since $u_\lambda\in{\rm int}\,C_+$). Let $\delta\in(0,1]$, $\rho=\max\{\|u_\lambda\|_\infty,m_\lambda+1\}$ and let $\hat{\xi}_\rho>0$ be as postulated by hypothesis $H_1(v)$. We set $m_\lambda^\delta=m_\lambda+\delta$
\begin{eqnarray*}
   && -\Delta_{p(z)}m_\lambda^\delta-\Delta_{q(z)}m_\lambda^\delta+[\xi(z)+\hat{\xi}_\rho](m_\lambda^\delta)^{p(z)-1}-
(m_\lambda^\delta)^{-\eta(z)} \\
   &\leq& [\xi(z)+\hat{\xi}_\rho]m_\lambda^{p(z)-1}+\chi(\delta) \mbox{ with }\chi(\delta)\to0^+ \mbox{ as }\delta\to0^+ \\
   &\leq& \lambda_0 f(z,m_\lambda)+\hat{\xi}_\rho m_\lambda^{p(z)-1}+\chi(\delta) \mbox{ (see \eqref{eq27}) } \\
   &=& \lambda f(z,m_\lambda)+\hat{\xi}_\rho m_\lambda^{p(z)-1}-[\lambda-\lambda_0]f(z,m_\lambda)+\chi(\delta) \\
   &\leq& \lambda f(z,m_\lambda)+\hat{\xi}_\rho m_\lambda^{p(z)-1}-[\lambda-\lambda_0]\hat{\mu}_\lambda+\chi(\delta)\\
  && \mbox{ with $\hat{\mu}_\lambda>0$ (see hypothesis $H_1(iv)$) } \\
  &\leq& \lambda f(z,u_\lambda)+\hat{\xi}_\rho u_\lambda^{p(z)-1} \mbox{ for }\delta>0 \mbox{ small } \mbox{ (see hypothesis $H_1(v)$)} \\
  &=& -\Delta_{p(z)}u_\lambda-\Delta_{q(z)}u_\lambda +[\xi(z)+\hat{\xi}_\rho]u_\lambda^{p(z)-1}-u_\lambda^{-\eta(z)} \mbox{( since $u_\lambda\in S_\lambda$),} \\
  &\Rightarrow& u_\lambda-m_\lambda^\delta \in D_+ \mbox{ for }\delta>0 \mbox{ small }\mbox{ (see Proposition \ref{prop3})},
\end{eqnarray*}
a contradiction to the definition of $m_\lambda>0$.

So, we have $\lambda^*\leq \lambda_0<\infty$.
\end{proof}

Next, we show that $\mathcal{L}$ is, in fact, an interval.

\begin{prop}\label{prop13}
  If hypotheses $H_0$, $H_1$ hold, $\lambda\in\mathcal{L}$ and $0<\mu<\lambda$, then $\mu\in \mathcal{L}$.
\end{prop}

\begin{proof}
  Since $\lambda\in\mathcal{L}$, we can find $u_\lambda\in S_\lambda\subseteq{\rm int}\,C_+$ (see Proposition \ref{prop11}). Then we have
\begin{eqnarray}\nonumber
  && -\Delta_{p(z)}u_\lambda-\Delta_{q(z)} u_\lambda +\xi(z)u_\lambda^{p(z)-1} \\ \nonumber
  &=& u_\lambda^{-\eta(z)}+\lambda f(z,u_\lambda) \\
  &\geq& u_\lambda^{-\eta(z)}+\mu f(z,u_\lambda) \mbox{ in }\Omega \label{eq28}\\ \nonumber
  && \mbox{ (recall that $0<\mu<\lambda$ and see hypothesis $H_1(i)$). }
\end{eqnarray}

Also we have
\begin{eqnarray}\nonumber
  && -\Delta_{p(z)}\overline{u}-\Delta_{q(z)}\overline{u}+\xi(z)\overline{u}^{p(z)-1} \\ \nonumber
  &=& \overline{u}^{-\eta(z)} \\
  &\leq& \overline{u}^{-\eta(z)}+\mu f(z,\overline{u}) \mbox{ in } \Omega \mbox{ (see hypothesis $H_1(i)$)}. \label{eq29}
\end{eqnarray}

From Proposition \ref{prop10} we know that $\overline{u}\leq u_\lambda$. So, we can define the following truncation of the reaction of problem $(P_\mu)$
\begin{equation}\label{eq30}
  \hat{\tau}_\mu(z,x)=\left\{
                        \begin{array}{ll}
                          \overline{u}^{-\eta(z)}+\mu f(z,\overline{u}(z)), & \hbox{ if }x<\overline{u}(z) \\
                           x^{-\eta(z)}+\mu f(z,x), & \hbox{ if }\overline{u}(z)\leq x\leq u_\lambda(z) \\
                           u_\lambda(z)^{-\eta(z)}+\mu f(z,u_\lambda(z)), & \hbox{ if } u_\lambda(z)<x.
                        \end{array}
                      \right.
\end{equation}

This is a Carath\' eodory function. We set $\hat{T}_\mu(z,x)=\displaystyle{\int_0^x \hat{\tau}(z,s)ds}$ and consider the $C^1$-functional $\hat{w}_\mu:W^{1,p(z)}(\Omega)\to \RR$ defined by
$$
\hat{w}_\mu(u)=\gamma_{p(z)}(u)+\int_\Omega \frac{1}{q(z)}|Du|^{q(z)}dz-\int_\Omega \hat{T}_\mu(z,u)dz \mbox{ for all }u\in W^{1,p(z)}(\Omega).
$$

As before (see the proof of Proposition \ref{prop9}), using the direct method of the calculus of variations and \eqref{eq28}, \eqref{eq29}, we can find $u_\mu\in W^{1,p(z)}(\Omega)$ such that
\begin{eqnarray*}
  && u_\mu\in K_{w_\lambda}\subseteq[\overline{u},u_\lambda]\cap{\rm int}\,C_+, \\
  &\Rightarrow& u_\mu\in S_\mu \mbox{ (see \eqref{eq30}), } \\
  &\Rightarrow& \mu\in\mathcal{L}.
\end{eqnarray*}
The proof is now complete.
\end{proof}

\begin{rem}
  As a byproduct of this proof we have that if $0<\mu<\lambda\in\mathcal{L}$ and $u_\lambda\in S_\lambda\subseteq{\rm int}\,C_+$ then we can find $u_\mu\in S_\mu\subseteq{\rm int}\,C_+$ such that
$$
u_\lambda-u_\mu\in C_+\setminus\{0\}.
$$

In fact we can improve this result using the strong comparison principle (see Proposition \ref{prop3}).
\end{rem}

\begin{prop}\label{prop14}
  If hypotheses $H_0$, $H_1$ hold, $0<\mu<\lambda\in\mathcal{L}$ and $u_\lambda\in S_\lambda\subseteq{\rm int}\,C_+$, then $\mu\in \mathcal{L}$ and there exists $u_\mu\in S_\mu\subseteq{\rm int}\,C_+$ such that $u_\lambda-u_\mu\in D_+$.
\end{prop}

\begin{proof}
  From Proposition \ref{prop13} and its proof, we already know that $\mu\in\mathcal{L}$ and we can find $u_\mu\in S_\mu\subseteq{\rm int}\,C_+$ such that
\begin{equation}\label{eq31}
  u_\mu\leq u_\lambda,\;u_\mu\not=u_\lambda.
\end{equation}

Let $\rho=\|u_\lambda\|_\infty$ and let $\hat{\xi}_\rho>0$ be as postulated by hypothesis  $H_1(v)$. We have
\begin{eqnarray}\nonumber
  && -\Delta_{p(z)}u_\mu-\Delta_{q(z)}u_\mu+[\xi(z)+\hat{\xi}_\rho]u_\mu^{p(z)-1}-u_\mu^{-\eta(z)} \\ \nonumber
  &=& \mu f(z,u_\mu)+\hat{\xi}_\rho u_\mu^{p(z)-1} \\ \nonumber
  &=& \lambda f(z,u_\mu)+\hat{\xi}_\rho u_\mu^{p(z)-1}-(\lambda-\mu)f(z,u_\mu) \\ \nonumber
  &\leq& \lambda f(z,u_\lambda)+\hat{\xi}_\rho u_\lambda^{p(z)-1} \mbox{ (see \eqref{eq31} and hypothesis $H_1(v)$)} \\
   &=& -\Delta_{p(z)}u_\lambda-\Delta_{q(z)}u_\lambda+[\xi(z)+\hat{\xi}_\rho]u_\lambda^{p(z)-1} \mbox{ in }\Omega. \label{eq32}
\end{eqnarray}

On account of hypothesis $H_1(iv)$ and since $u_\mu\in{\rm int}\,C_+$, we have
$$
0<\hat{\mu}_\mu\leq(\lambda-\mu)f(z,u_\mu(z)) \mbox{ for a.a. }z\in\Omega.
$$

Hence from \eqref{eq32} and Proposition \ref{prop3}, it follows that
$$
u_\lambda-u_\mu\in D_+.
$$
The proof is now complete.
\end{proof}

\begin{prop}\label{prop15}
  If hypotheses $H_0$, $H_1$ hold and $\lambda\in(0,\lambda^*)$, then problem $(P_\lambda)$ admits at least two positive solutions
$$
u_0,\hat{u}\in{\rm int}\,C_+,\ u_0\leq \hat{u},\ u_0\not=\hat{u}.
$$
\end{prop}

\begin{proof}
  Let $0<\lambda<\vartheta<\lambda^*$. We know that $\vartheta\in\mathcal{L}$ (see Proposition \ref{prop13}). Moreover, according to Proposition \ref{prop14}, we can find $u_\vartheta\in S_\vartheta\subseteq{\rm int}\,C_+$ and $u_0\in S_\lambda\subseteq{\rm int}\,C_+$ such that
\begin{equation}\label{eq33}
  u_\vartheta-u_0\in{\rm int}\,C_+.
\end{equation}

We introduce the Carath\' eodory function $d_\lambda(z,x)$ defined by
\begin{equation}\label{eq34}
  d_\lambda(z,x)=\left\{
                   \begin{array}{ll}
                     u_0(z)^{-\eta(z)}+\lambda f(z,u_0(z)), & \hbox{ if }x\leq u_0(z) \\
                     x^{-\eta(z)}+\lambda f(z,x), & \hbox{ if } u_0(z)<x.
                   \end{array}
                 \right.
\end{equation}

Also we consider the following truncation of $d_\lambda(z,\cdot)$:
\begin{equation}\label{eq35}
  \hat{d}_\lambda(z,x)=\left\{
                         \begin{array}{ll}
                           d_\lambda(z,x), & \hbox{ if }x\leq u_\vartheta(z) \\
                           d_\lambda(z,u_\vartheta(z)), & \hbox{ if }u_\vartheta(z)<x.
                         \end{array}
                       \right.
\end{equation}

This  is also a Carath\'eodory function. We set
$$
D_\lambda(z,x)=\int_0^x d_\lambda(z,s)ds \mbox{ and } \hat{D}_\lambda(z,x)=\int_0^x \hat{d}_\lambda(z,s)ds.
$$

Then we consider the $C^1$-functionals $v_\lambda,\hat{v}_\lambda:W^{1,p(z)}(\Omega)\to\RR$ defined by
\begin{eqnarray*}
  && v_\lambda(u)=\gamma_{p(z)}(u)+\int_\Omega \frac{1}{q(z)}|Du|^{q(z)}-\int_\Omega D_\lambda(z,u)dz, \\
  && \hat{v}_\lambda(u)=\gamma_{p(z)}(u)+\int_\Omega \frac{1}{q(z)} |Du|^{q(z)}dz-\int_\Omega \hat{D}_\lambda(z,u) dz \mbox{ for all }u\in W^{1,p(z)}(\Omega).
\end{eqnarray*}

From \eqref{eq34}, \eqref{eq35} and the anisotropic regularity theory, we have
\begin{equation}\label{eq36}
  K_{v_\lambda}\subseteq[u_0)\cap{\rm int}\,C_+,
\end{equation}
\begin{equation}\label{eq37}
  K_{\hat{v}_\lambda}\subseteq[u_0,u_\vartheta]\cap{\rm int}\,C_+.
\end{equation}

On account of \eqref{eq34} and \eqref{eq36}, we see that we may assume that
\begin{equation}\label{eq38}
  K_{v_\lambda}\cap[u_0,u_\vartheta]=\{u_0\}.
\end{equation}

Otherwise we already have a second positive smooth solution bigger that $u_0$ (see \eqref{eq34}, \eqref{eq36}). Also from \eqref{eq34} and \eqref{eq35} we see that
\begin{equation}\label{eq39}
  v_\lambda\Big|_{[0,u_\vartheta]}=\hat{v}_\lambda\Big|_{[0,u_\vartheta]},\; v'_\lambda\Big|_{[0,u_\vartheta]}=\hat{v}'_\lambda\Big|_{[0,u_\vartheta]}.
\end{equation}

The functional $\hat{v}_\lambda(\cdot)$ is coercive and sequentially weakly lower semicontinuous. So, we can find $\tilde{u}_0\in W^{1,p(z)}(\Omega)$ such that
\begin{eqnarray}\nonumber
  && \hat{v}_\lambda(\tilde{u}_0)=\min\left\{\hat{v}_\lambda(u):\:u\in W^{1,p(z)}(\Omega)\right\}, \\ \nonumber
  &\Rightarrow& \tilde{u}_0\in K_{\hat{v}_\lambda}\subseteq[u_0,u_\vartheta]\cap{\rm int}\,C_+ \mbox{ (see \eqref{eq37}), } \\ \nonumber
  &\Rightarrow& \tilde{u}_0=u_0 \mbox{ (see \eqref{eq38}, \eqref{eq39}), } \\ \nonumber
  &\Rightarrow& u_0 \mbox{ is a local $C^1(\overline{\Omega})$-minimizer of $v_\lambda(\cdot)$ (see \eqref{eq33},\eqref{eq39}), } \\
  &\Rightarrow& u_0 \mbox{ is a local $W^{1,p(z)}(\Omega)$-minimizer of $v_\lambda(\cdot)$ } \label{eq40} \\ \nonumber
  && \mbox{ (see Gasinski \& Papageorgiou \cite[Proposition 3.3]{9Gas-Pap}). }
\end{eqnarray}

From \eqref{eq36} and \eqref{eq34} it is clear that we may assume that
\begin{equation}\label{eq41}
  K_{v_\lambda} \mbox{ is finite. }
\end{equation}
Indeed,
otherwise we already have a whole sequence of positive smooth solutions all bigger than $u_0$ and so we are done.

From \eqref{eq40}, \eqref{eq41} and Theorem 5.7.6 of Papageorgiou, R\u adulescu \& Repov\v s \cite[p. 449]{15Pap-Rad-Rep}, we know that we can find $\rho\in(0,1)$ small such that
\begin{equation}\label{eq42}
  v_\lambda(u_0)<\inf\left\{v_\lambda(u):\:\|u-u_0\|=\rho\right\}=m_\lambda.
\end{equation}

Also hypothesis $H_1(ii)$ and \eqref{eq34} imply that if $u\in {\rm int}\,C_+$, then
\begin{equation}\label{eq43}
  v_\lambda(tu)\to-\infty \mbox{ as }t\to+\infty.
\end{equation}

\smallskip {\it
Claim}: $v_\lambda(\cdot)$ satisfies the $C$-condition.

We consider a sequence $\{u_n\}_{n\geq1}\subseteq W^{1,p(z)}(\Omega)$ such that
\begin{equation}\label{eq44}
  |v_\lambda(u_n)|\leq C_7 \mbox{ for some } C_7>0, \mbox{ all }n\in\NN,
\end{equation}
\begin{equation}\label{eq45}
  (1+\|u_n\|)v'_\lambda(u_n)\to0 \mbox{ in }W^{1,p(z)}(\Omega)^* \mbox{ as }n\to\infty.
\end{equation}

From \eqref{eq45} we have
\begin{equation}\label{eq46}
  \left|\langle A_{p(z)}(u_n),h\rangle+\langle A_{q(z)}(u_n),h\rangle+\int_\Omega \xi(z)|u_n|^{p(z)-2}u_n hdz-\int_\Omega d_\lambda(z,u_n)hdz\right|\leq \frac{\varepsilon_n\|h\|}{1+\|u_n\|}
\end{equation}
for all $h\in W^{1,p(z)}(\Omega)$, with $\varepsilon_n\to0^+$.

In \eqref{eq46}, we choose $h=-u_n^-\in W^{1,p(z)}(\Omega)$ and obtain
\begin{eqnarray}\nonumber
  && \rho_{p(z)}(Du_n^-)+\rho_{q(z)}(Du_n^-)+\int_\Omega \xi(z)(u_n^-)^{p(z)}dz\leq C_8 \\ \nonumber
  && \mbox{ for some } C_8>0, \mbox{ all }n\in\NN \mbox{ (see \eqref{eq34}), } \\
  &\Rightarrow& \{u_n^-\}_{n\geq1}\subseteq W^{1,p(z)}(\Omega) \mbox{ is bounded.} \label{eq47}
\end{eqnarray}

Next, we choose in \eqref{eq46} $h=u_n^+\in W^{1,p(z)}(\Omega)$ and we obtain
\begin{equation}\label{eq48}
  -\rho_{p(z)}(Du_n^+)-\rho(Du_n^+)-\int_\Omega \xi(z)(u_n^+)^{p(z)}dz +\int_\Omega d_\lambda(z,u_n^+)u_n^+dz\leq \varepsilon_n
\end{equation}
for all $n\in\NN$.

From \eqref{eq44} and \eqref{eq47}, we have
\begin{eqnarray}
  && \rho_{p(z)}(Du_n^+)+\rho_{q(z)}(Du_n^+)+\int_\Omega \xi(z)(u_n^+)^{p(z)}dz-\int_\Omega p_+ D_\lambda(z,u_n^+)dz\leq C_9 \label{eq49} \\ \nonumber
  && \mbox{ for some $C_9>0$, all $n\in\NN$ (recall $q_-<p_-$). }
\end{eqnarray}

We add \eqref{eq48} and \eqref{eq49}. Then
\begin{eqnarray}
  \int_\Omega I_\lambda(z,u_n^+)dz&\leq& C_{10}  \label{eq50}\\ \nonumber
  && \mbox{ for some $C_{10}>0$, all $n\in\NN$ (see \eqref{eq34}). }
\end{eqnarray}

Suppose that $\{u_n^+\}_{n\geq1}\subseteq W^{1,p(z)}(\Omega)$ is not bounded. So, we may assume that
\begin{equation}\label{eq51}
  \|u_n^+\|\to+\infty \mbox{ as } n\to \infty.
\end{equation}

We set $y_n=\frac{u_n^+}{\|u_n^+\|}$ for all $n\in\NN$. Then $\|y_n\|=1$, $y_n\geq0$ for all $n\in\NN$. So, we may assume that
\begin{equation}\label{eq52}
  y_n\overset{w}{\to}y \mbox{ in } W^{1,p(z)}(\Omega) \mbox{ and }y_n\to y \mbox{ in } L^{p(z)}(\Omega), \ y\geq0.
\end{equation}

Initially we assume that $y\not\equiv0$. Let $\hat{\Omega}=\{z\in\Omega:\:y(z)>0\}$. From \eqref{eq52} we see that $|\hat{\Omega}|_N>0$ (by $|\cdot|_N$ we denote the Lebesgue measure on $\RR^N$) and
$$
u_n^+(z)\to+\infty \mbox{ for a.a. }z\in\hat{\Omega}.
$$

Hypothesis $H_1(ii)$ implies
\begin{eqnarray}\nonumber
   && \frac{F(z,u_n^+(z))}{\|u_n^+\|^{p_+}}=\frac{F(z,u_n^+(z))}{u_n^+(z)^{p_+}}y_n(z)^{p_+}\to+\infty \mbox{ for a.a. }z\in\hat{\Omega}, \\ \nonumber
   &\Rightarrow& \int_{\hat{\Omega}} \frac{F(z,u_n^+)}{\|u_n^+\|^{p_+}}dz\to+\infty \mbox{ (by Fatou's lemma), }\\
   &\Rightarrow& \int_\Omega \frac{F(z,u_n^+)}{\|u_n^+\|^{p_+}}dz\to+\infty \mbox{ (since $F\geq0$, see hypothesis $H_1(i)$). } \label{eq53}
\end{eqnarray}

From \eqref{eq44} and \eqref{eq47}, we have
\begin{eqnarray}\nonumber
  && -\gamma_{p(z)}(u_n^+)-\int_\Omega \frac{1}{q(z)}|Du_n^+|^{q(z)}dz-\int_\Omega \frac{\xi(z)}{p(z)}(u_n^+)^{p(z)}dz+\int_\Omega D_\lambda(z,u_n^+)dz\leq C_{11} \\ \nonumber
  && \mbox{ for some $C_{11}>0$, all $n\in\NN$, } \\
  &\Rightarrow& \int_\Omega \frac{\lambda F(z,u_n^+)}{\|u_n^+\|^{p_+}}dz\leq C_{12} \label{eq54} \\ \nonumber
  && \mbox{ for some $C_{12}>0$, all $n\in\NN$ (see \eqref{eq51} and \eqref{eq34}). }
\end{eqnarray}

Comparing \eqref{eq53} and \eqref{eq54}, we have a contradiction.

Now we assume that $y\equiv0$. We consider the $C^1$-functional $\tilde{v}_\lambda:W^{1,p(z)}(\Omega)\to \RR$ defined by
$$
\tilde{v}_\lambda(u)=\frac{1}{p_+}\left[\int_\Omega |Du|^{p(z)}dz+\int_\Omega \xi(z)|u|^{p(z)}dz\right]-\int_\Omega D_\lambda(z,u)dz
$$
for all $u\in W^{1,p(z)}(\Omega)$.

Evidently we have
\begin{equation}\label{eq55}
  \tilde{v}_\lambda\leq v_\lambda.
\end{equation}

Let $k_n(t)=\tilde{v}_\lambda(tu_n^+)$ for all $t\in[0,1]$, all $n\in\NN$. We can find $t_n\in[0,1]$ such that
\begin{equation}\label{eq56}
  k_n(t_n)=\max_{0\leq t\leq1} k_n(t).
\end{equation}

For $\hat{\eta}>0$ let $w_n=(2\hat{\eta})^{\frac{1}{p(z)}}y_n$ $n\in\NN$. Then
\begin{eqnarray}\nonumber
  && w\to0 \mbox{ in } L^{p(z)}(\Omega) \mbox{ (see \eqref{eq52} and recall that $y\equiv0$), } \\
  &\Rightarrow& \int_\Omega D_\lambda(z,w_n)dz\to0 \mbox{ as }n\to\infty. \label{eq57}
\end{eqnarray}

On account of \eqref{eq51}, we see that we can find $n_0\in\NN$ such that
\begin{equation}\label{eq58}
  (2\hat{\eta})^{\frac{1}{p(z)}}\frac{1}{\|u_n^+\|}\leq1 \mbox{ for all }n\geq n_0.
\end{equation}

From \eqref{eq56} and \eqref{eq58}, we have
$$
k_n(t_n)\geq k_n\left(\frac{(2\hat{\eta})^{1/p(z)}}{\|u_n^+\|}\right),
$$
\begin{eqnarray*}
  \Rightarrow \tilde{v}_\lambda(t_n u_n^+) &\geq& \tilde{v}_\lambda\left((2\hat{\eta})^{1/p(z)}y_n\right)=\tilde{v}_\lambda(w_n), \\
  \Rightarrow \tilde{v}_\lambda(t_n u_n^+) &\geq& \frac{2\hat{\eta}}{p_+}\left[\rho_{p(z)}(Dy_n)+\rho_{p(z),\xi}(y_n)\right]-\int_\Omega D_\lambda(z,w_n)dz \\
  &\geq& \frac{2\hat{\eta}}{p_+} C_{13}\|y_n\|-\int_\Omega D_\lambda(z,w_n)dz \ \ \mbox{for some $C_{13}>0$}\\
  &\geq& \frac{\hat{\eta}}{p_+}C_{13} \mbox{ for all }n\geq n_1\geq n_0 \\
  && \mbox{ (see \eqref{eq57} and recall that $\|y_n\|=1$). }
\end{eqnarray*}

But $\hat{\eta}>0$ is arbitrary. So, it follows that
\begin{equation}\label{eq59}
  \tilde{v}_\lambda(t_n u_n^+)\to+\infty \mbox{ as }n\to\infty.
\end{equation}

We have
$$
0\leq t_n u_n^+\leq u_n^+ \mbox{ for all }n\in\NN.
$$

Hence hypothesis $H_1(iii)$ implies that
\begin{eqnarray}
   && \int_\Omega I_\lambda (z,t_n u_n^+)dz\leq \int_\Omega I_\lambda(z,u_n^+)dz +\|\hat{\vartheta}_\lambda\|_1\leq C_{14} \label{eq60} \\ \nonumber
  && \mbox{ for some $C_{14}>0$, all $n\in\NN$ (see \eqref{eq50}). }
\end{eqnarray}

We know that
$$
\tilde{v}_\lambda(0)=0 \mbox{ and } \tilde{v}_\lambda(u_n)\leq C_7 \mbox{ for all }n\in\NN \mbox{ (see \eqref{eq44} and \eqref{eq55}). }
$$

Then from \eqref{eq59} it follows that
$$
t_n\in(0,1) \mbox{ for all }n\geq n_2.
$$

So, from \eqref{eq56} we have
\begin{eqnarray}\nonumber
   0&=& t_n\frac{d}{dt}\tilde{v}_\lambda(tu_n^+)\Big|_{t=t_n} \\ \nonumber
    &=& \langle \tilde{v}'_\lambda(t_n u_n^+),t_n u_n^+\rangle \mbox{ for all }n\geq n_2 \mbox{ (by the chain rule), } \\ \nonumber
  \Rightarrow && \rho_{p(z)}(D(t_n u_n^+))+\rho_{p(z),\xi}(t_n u_n^+)=\int_\Omega d_\lambda (z,t_n u_n^+)(t_n u_n^+)dz \mbox{ for all }n\geq n_2, \\ \nonumber
  \Rightarrow && \rho_{p(z)}(D(t_n u_n^+))+\rho_{p(z),\xi}(t_n u_n^+)-\int_\Omega p_+D_\lambda(z,t_n u_n^+)dz\leq C_{15} \\ \nonumber
  && \mbox{ for some $C_{15}>0$, all $n\geq n_2$ (see \eqref{eq60} and \eqref{eq34}), } \\
  \Rightarrow&& p_+ \tilde{v}_\lambda(t_n u_n^+)\leq C_{15} \mbox{ for all } n\geq n_2. \label{eq61}
\end{eqnarray}

Comparing \eqref{eq59} and \eqref{eq61} we have a contradiction.

Therefore we infer that
\begin{eqnarray*}
  && \{u_n^+\}_{n\geq1} \subseteq W^{1,p(z)}(\Omega) \mbox{ is bounded, } \\
  &\Rightarrow& \{u_n\}_{n\geq1}\subseteq W^{1,p(z)}(\Omega) \mbox{ is bounded (see \eqref{eq47})}.
\end{eqnarray*}

So, we may assume that
\begin{equation}\label{eq62}
  u_n\overset{w}{\to}u \mbox{ in } W^{1,p(z)}(\Omega)\mbox{ and } u_n\to u \mbox{ in } L^{r(z)}(\Omega) \mbox{ as } n\to\infty.
\end{equation}

In \eqref{eq46} we choose $h=u_n-u\in W^{1,p(z)}(\Omega)$, pass to the limit as $n\to\infty$ and use \eqref{eq60}. Then reasoning as in the proof of Proposition \ref{prop6}, using Proposition \ref{prop2}, we obtain that
\begin{eqnarray*}
  && u_n\to u\mbox{ in }W^{1,p(z)}(\Omega), \\
  &\Rightarrow& v_\lambda(\cdot) \mbox{ satisfies the $C$-condition. }
\end{eqnarray*}

This proves the claim.

On account of \eqref{eq42}, \eqref{eq43} and the claim, we can apply the mountain pass theorem and find $\hat{u}\in W^{1,p(z)}(\Omega)$ such that
\begin{equation}\label{eq63}
  \hat{u} \in K_{v_\lambda}\subseteq[u_0)\cap{\rm int}\,C_+ \mbox{ and }m_\lambda\leq v_\lambda(\hat{u}).
\end{equation}

From \eqref{eq63}, \eqref{eq42} and \eqref{eq34}, we obtain
$$
  \hat{u}\in S_\lambda\subseteq{\rm int}\,C_+, \ u_0\leq \hat{u},\ u_0\not=\hat{u},
$$
which concludes the proof.
\end{proof}

We introduce the Carath\'eodory function $\hat{\beta}_\lambda(z,x)$ defined by
\begin{equation}\label{eq64}
  \hat{\beta}_\lambda(z,x)=\left\{
                             \begin{array}{ll}
                               \overline{u}(z)^{-\eta(z)}+\lambda f(z,\overline{u}(z)), & \hbox{ if }x\leq \overline{u}(z) \\
                               x^{-\eta(z)}+\lambda f(z,x), & \hbox{ if }\overline{u}(z)<x.
                             \end{array}
                           \right.
\end{equation}

We set $\hat{B}_\lambda(z,x)=\displaystyle{\int_0^x}\hat{\beta}(z,s)ds$ and consider the $C^1$-functional $\hat{\varphi}_\lambda:W^{1,p(z)}(\Omega)\to\RR$ defined by
$$
\hat{\varphi}_\lambda(u)=\gamma_{p(z)}(u)+\int_\Omega \frac{1}{q(z)}|Du|^{q(z)}dz-\int_\Omega \hat{B}_\lambda(z,u)dz
$$
for all $u\in W^{1,p(z)}(\Omega)$.

Using this functional, we can establish the admissibility of the critical parameter value $\lambda^*>0$.

\begin{prop}\label{prop16}
  If hypotheses $H_0$, $H_1$ hold, then $\lambda^*\in \mathcal{L}$.
\end{prop}

\begin{proof}
  Let $\{\lambda_n\}_{n\geq1}\subseteq(0,\lambda^*)\subseteq\mathcal{L}$ and assume that $\lambda_n\to\lambda^*$. We can find $u_n\in S_\lambda\subseteq{\rm int}\,C_+$ $n\in\NN$ such that
\begin{eqnarray}
  \hat{\varphi}_{\lambda_n}(u_n) &\leq& \hat{\varphi}_{\lambda_n}(\overline{u})<0 \mbox{ for all }n\in\NN \label{eq65}\\ \nonumber
  && \mbox{ (see the proof of Proposition \ref{prop13} and \eqref{eq64}). }
\end{eqnarray}

Also we have
\begin{equation}\label{eq66}
  \hat{\varphi}'_{\lambda_n}(u_n)=0 \mbox{ for all }n\in\NN.
\end{equation}

Using \eqref{eq65}, \eqref{eq66} and reasoning as in the claim in the proof of Proposition \ref{prop15}, we obtain
\begin{eqnarray*}
  && u_n\to u^* \mbox{ in }W^{1,p(z)}(\Omega),\overline{u}\leq u^* \mbox{ (see Proposition \ref{prop10})} \\
  &\Rightarrow& u^*\in S_{\lambda^*}\subseteq{\rm int}\,C_+ \mbox{ and so }\lambda^*\in \mathcal{L}.
\end{eqnarray*}
The proof is now complete.
\end{proof}

According to the above proposition, we have
$$
\mathcal{L}=(0,\lambda^*].
$$

\section{Minimal positive solutions}

Recall that a set $S\subseteq W^{1,p(z)}(\Omega)$ is said to be downward directed, if for all $u_1,u_2\in S$, we can find $u\in S$ such that $u\leq u_1$, $u\leq u_2$.

As in the proof of Proposition 18 of Papageorgiou, R\u adulescu \& Repov\v s \cite{16Pap-Rad-Rep}, we prove that for every $\lambda\in\mathcal{L}$, the solution set $S_\lambda\subseteq{\rm int}\,C_+$ is downward directed. We will show that $S_\lambda$ has a minimal element.

\begin{prop}\label{prop17}
  If hypotheses $H_0$, $H_1$ hold and $\lambda\in\mathcal{L}=(0,\lambda^*]$, then problem $(P_\lambda)$ has a smallest positive solution $u^*_\lambda\in S_\lambda\subseteq{\rm int}\,C_+$
 (that is, $u_\lambda^*\leq u$ for all $u\in S_\lambda$).
\end{prop}

\begin{proof}
  On account of Lemma 3.10 of Hu \& Papageorgiou \cite[p. 176]{12Hu-Pap}, we can find $\{u_n\}_{n\geq1}\subseteq S_\lambda$ decreasing (since $S_\lambda$ is downward directed) such that
$$
\inf_{n\geq1} u_n=\inf S_\lambda.
$$

We have
\begin{eqnarray}
  && \hat{\varphi}'_\lambda(u_n)=0 \mbox{ for all }n\in\NN, \label{eq67} \\
  && \overline{u}\leq u_n\leq u_1 \mbox{ for all }n\in \NN. \label{eq68}
\end{eqnarray}

From \eqref{eq67} and \eqref{eq68}, it follows that
$$
\{u_n\}_{n\geq1}\subseteq W^{1,p(z)}(\Omega) \mbox{ is bounded. }
$$

Then as in the proof of Proposition \ref{prop13} and using the fact that $\{u_n\}_{n\geq1}$ is decreasing, we obtain
\begin{eqnarray*}
  && u_n\to u_\lambda^* \mbox{ in }W^{1,p(z)}(\Omega), \\
  &\Rightarrow& u_\lambda^*\in S_\lambda\subseteq{\rm int}\,C_+, u_\lambda^*=\inf S_\lambda,
\end{eqnarray*}
which concludes the proof.
\end{proof}

Next we examine the map $\lambda\mapsto u_\lambda^*$ from $\mathcal{L}=(0,\lambda^*]$ into $C^1(\overline{\Omega})$.

\begin{prop}\label{prop18}
If hypotheses $H_0$, $H_1$ hold, then
\begin{itemize}
  \item[$(a)$] the map $\lambda\mapsto u_\lambda^*$ from $\mathcal{L}=(0,\lambda^*]$ into $C^1(\overline{\Omega})$ is strictly increasing in the sense that
$$
0<\mu<\lambda\leq \lambda^*\Rightarrow u_\lambda^*-u_\mu^* \in D_+;
$$
  \item[$(b)$] $\lambda\mapsto u_\lambda^*$ is left continuous from $\mathcal{L}$ into $C^1(\overline{\Omega})$.
\end{itemize}
\end{prop}

\begin{proof}
  $(a)$ Let $0<\mu<\lambda\leq \lambda^*$. On account of Proposition \ref{prop14}, we can find $u_\mu\in S_\mu\subseteq{\rm int}\,C_+$ such that
\begin{eqnarray*}
   && u_\lambda^*-u_\mu\in D_+, \\
  &\Rightarrow& u_\lambda^*-u_\mu^* \in D_+.
\end{eqnarray*}
$(b)$ Let $\lambda_n\to \lambda^-$, $\lambda_n\in\mathcal{L}$ for all $n\in\NN$. From Proposition \ref{prop10} and part $(a)$ of this proposition, we have
$$
\overline{u}\leq u_\lambda^*\leq u_{\lambda_1}^* \mbox{ for all }n\in\NN.
$$

It follows that
$$
\{u_n^*=u_{\lambda_n}^*\}_{n\geq1}\subseteq W^{1,p(z)}(\Omega) \mbox{ is bounded. }
$$

Then Proposition 3.1 of Gasinski \& Papageorgiou \cite{9Gas-Pap} implies that
$$
\|u_n^*\|_\infty\leq C_{16} \mbox{ for some }C_{16}>0, \mbox{ all }n\in\NN.
$$

Using Theorem 1.3 of Fan \cite{8Fan} (see also Lieberman \cite{13Lieberman}), we see that we can find $\alpha\in(0,1)$ and $C_{17}>0$ such that
\begin{equation}\label{eq69}
  u_n^*\in C^{1,\alpha}(\overline{\Omega}) \mbox{ and } \|u_n^*\|_{C^{1,\alpha}(\overline{\Omega})}\leq C_{17} \mbox{ for all }n\in \NN.
\end{equation}

From \eqref{eq69}, the compact embedding of $C^{1,\alpha}(\overline{\Omega})$ into $C^1(\overline{\Omega})$ and the monotonicity of the sequence $\{u_n^*\}_{n\geq1}$, we have that
\begin{equation}\label{eq70}
  u_n^*\to \hat{u}_\lambda^* \mbox{ in }C^1(\overline{\Omega}).
\end{equation}

If $\hat{u}_\lambda^*\not=u_\lambda^*$, then there exists $z_0\in\overline{\Omega}$ such that
\begin{eqnarray*}
  && u_\lambda^*(z_0)<\hat{u}_\lambda^*(z_0), \\
  &\Rightarrow& u_\lambda^*(z_0)<u_n^*(z_0) \mbox{ for all }n\geq n_0 \mbox{ (see \eqref{eq70}), }
\end{eqnarray*}
contradicting part $(a)$. So $\hat{u}_\lambda^*=u_\lambda^*$ and we have the left continuity of the map $\lambda\mapsto u_\lambda^*$.
\end{proof}

The next theorem summarizes our main contributions in this paper concerning problem $(P_\lambda)$.

\begin{thm}\label{th1}
  If hypotheses $H_0$, $H_1$ hold, then there exists $\lambda^*<\infty$ such that
\begin{itemize}
  \item[(a)] for all $\lambda\in(0,\lambda^*)$, problem $(P_\lambda)$ has at least two positive solutions $u_0,\hat{u}\in{\rm int}\,C_+$, $u_0\leq \hat{u}$, $u_0\not=\hat{u}$;
  \item[(b)] for $\lambda=\lambda^*$, problem $(P_\lambda)$ has at least one positive solution $u^*\in{\rm int}\,C_+$;
  \item[(c)] for all $\lambda>\lambda^*$, problem $(P_\lambda)$ has no positive solutions;
  \item[(d)] for all $\lambda\in\mathcal{L}=(0,\lambda^*]$, problem $(P_\lambda)$ has a smallest positive solution $u_\lambda^*\in{\rm int}\,C_+$ and the map $\lambda\mapsto u_\lambda^*$ is strictly increasing and left-continuous from $\mathcal{L}$ into $C^1(\overline{\Omega})$.
\end{itemize}
\end{thm}

\subsection*{Acknowledgments} The authors were supported by the Slovenian Research Agency grants
P1-0292, J1-8131, N1-0064, N1-0083, and N1-0114. We wish to thank the anonymous reviewers for their remarks and constructive criticisms that helped us improve the presentation.

\end{document}